\documentclass[reqno]{amsart}
\usepackage{enumerate}
\usepackage{amsmath}%
\usepackage{amsfonts}%
\usepackage{amssymb}%
\usepackage{graphicx}
\usepackage{pgfplots}
\usepackage{mathrsfs}
\usepackage{hyperref}
\usepackage[margin=1in]{geometry}

\newtheorem{theorem}{Theorem}
\theoremstyle{plain}

\newtheorem{claim}[theorem]{Claim}

\newtheorem{construction}[theorem]{Construction}

\newtheorem{definition}[theorem]{Definition}

\newtheorem{fact}[theorem]{Fact}
\newtheorem{lemma}[theorem]{Lemma}

\newtheorem{proposition}[theorem]{Proposition}

\numberwithin{equation}{section}
\numberwithin{theorem}{section}
\numberwithin{case}{section}

\numberwithin{subcase}{case}

\def\A{\mathcal{A}}

\def\C{\mathcal{C}}

\def\F{\mathcal{F}}

\def\K{\mathcal{K}}
\def\R{\mathcal{R}}
\def \L{\mathcal{L}}
\def\M{\mathcal{M}}

\def\cP{\mathcal{P}}
\def\T{\mathcal{T}}

\def \a{\alpha}
\def \b{\beta}
\def \e{\epsilon}

\def \r{\gamma}

\def \t{\mathbf{t}}
\def \bfu{\mathbf{u}}

\def \bfv{\mathbf{v}}
\def\eg{\emph{e.g.}}
\def\ex{\text{ex}}

\begin{document}
\title[Vertex degree thresholds for tiling $3$-graphs]{Minimum vertex degree thresholds for tiling complete $3$-partite $3$-graphs}
\thanks{
The first author is supported by FAPESP (Proc. 2013/03447-6, 2014/18641-5, 2015/07869-8).
The third author is partially supported by NSA grant H98230-12-1-0283 and NSF grant DMS-1400073.}
\author{Jie Han}
\address{Instituto de Matem\'{a}tica e Estat\'{\i}stica, Universidade de S\~{a}o Paulo, Rua do Mat\~{a}o 1010, 05508-090, S\~{a}o Paulo, Brazil}
\email[Jie Han]{jhan@ime.usp.br}
\author{Chuanyun Zang}
\author{Yi Zhao}
\address
{Department of Mathematics and Statistics, Georgia State University, Atlanta, GA 30303} 
\email[Chuanyun Zang]{czang1@student.gsu.edu}
\email[Yi Zhao]{yzhao6@gsu.edu}

\date{\today}


\begin{abstract}
Given positive integers $a\leq b \leq c$, let $K_{a,b,c}$ be the complete 3-partite 3-uniform hypergraph with three parts of sizes $a,b,c$. Let $H$ be a 3-uniform hypergraph on $n$ vertices where $n$ is divisible by $a+b+c$. We asymptotically determine the minimum vertex degree of $H$ that guarantees a perfect $K_{a, b, c}$-tiling, that is, a spanning subgraph of $H$ consisting of vertex-disjoint copies of $K_{a, b, c}$. This partially answers a question of Mycroft, who proved an analogous result with respect to codegree for $r$-uniform hypergraphs for all $r\ge 3$. Our proof uses a lattice-based absorbing method, the concept of fractional tiling, and a recent result on shadows for 3-graphs.
\end{abstract}

\maketitle

\section{Introduction}

Given $r\ge 2$, an $r$-uniform hypergraph (in short, \emph{$r$-graph}) consists of a vertex set $V$ and an edge set $E\subseteq \binom{V}{r}$, that is, every edge is an $r$-element subset of $V$. Given an $r$-graph $H$ with a set $S$ of $d$ vertices, where $1 \le d\le r-1$, we define $\deg_{H}(S)$ to be the number of edges containing $S$ (the subscript $H$ is omitted if it is clear from the context). The \emph{minimum $d$-degree} $\delta_d(H)$ of $H$ is the minimum of $\deg_{H}(S)$ over all $d$-vertex sets $S$ in $H$. The minimum 1-degree is also referred as the \emph{minimum vertex degree}.

Given two $r$-graphs $F$ and $H$, an \emph{$F$-tiling} (also known as \emph{$F$-packing}) of $H$ is a collection of vertex-disjoint copies of $F$ in $H$. An $F$-tiling is called a \emph{perfect $F$-tiling} (or an \emph{$F$-factor}) of $H$ if it covers all the vertices of $H$. An obvious necessary condition for $H$ to contain an $F$-factor is $|V(F)|\mid |V(H)|$. Given an integer $n$ that is divisible by $|V(F)|$,  we define the \emph{tiling threshold} $t_d(n, F)$ to be the smallest integer $t$ such that every $r$-graph $H$ of order $n$ with $\delta_d(H)\ge t$ contains an $F$-factor.

As a natural extension of the matching problem, tiling has been intensively studied in the past two decades (see survey \cite{KuOs-survey}). Much work has been done on graphs ($r=2$), see \eg, \cite{HaSz, AY96, KSS-AY, KuOs09}. In particular, K\"uhn and Osthus \cite{KuOs09} determined $t_1(n, F)$, for any graph $F$, up to an additive constant. Tiling problems become much harder for hypergraphs ($r\ge 3$).
For example, despite efforts from many researchers \cite{AFHRRS, CzKa, Han16_mat, Khan1, Khan2, KOT, RRS09, TrZh13, TrZh16}, we still do not know the vertex degree threshold for a perfect matching in $r$-graphs for arbitrary $r$.

Other than the matching problem, only a few tiling thresholds are known (see survey \cite{zsurvey})
Let $K_4^3$ denote the complete 3-graph on four vertices, and let $K_4^3-e$ denote the (unique) 3-graph on four vertices with three edges. Recently Lo and Markstr\"om \cite{LM1} proved that $t_2(n, K_4^3)=(1+o(1))3n/4$, and Keevash and Mycroft \cite{KM1} determined the exact value of $t_2(n, K_4^3)$  for sufficiently large $n$. In \cite{LM2}, Lo and Markstr\"om proved that $t_2(n, K_4^3 - e)=(1+o(1))n/2$. Let
$\C_4^3$ be the unique 3-graph on four vertices with two edges (this 3-graph was denoted by $K_4^3 - 2e$ in \cite{CDN}, and by $Y$ in \cite{HZ1}). K\"uhn and Osthus \cite{KO} showed that $t_2(n, \C_4^3)=(1+o(1))n/4$, and Czygrinow, DeBiasio and Nagle \cite{CDN} determined $t_2(n, \C_4^3)$ exactly for large $n$.
Recently Mycroft \cite{My14} determined $t_{r-1}(n, F)$ asymptotically for many $r$-partite $r$-graphs $F$ including all complete $r$-partite $r$-graphs and loose cycles.

There are fewer tiling results on vertex degree conditions. Lo and Markstr\"om \cite{LM1} determined $t_1(n, K_3^3(m))$ and $t_1(n, K_4^4(m))$  asymptotically, where $K_r^r(m)$ denotes the complete $r$-partite $r$-graph with $m$ vertices in each part. Recently Han and Zhao \cite{HZ3} and independently Czygrinow \cite{Czy14} determined $t_1(n, \C_4^3)$ exactly for sufficiently large $n$.
In this paper we extend these results by determining $t_{1}(n, K)$ asymptotically for all complete $3$-partite $3$-graphs $K$, and thus partially answer a question of Mycroft \cite{My14talk}.

Given $a\le b\le c$, let $d=\gcd(b-a,c-b)$ and define
\begin{equation}\label{absbound}
f(a,b,c) :=
\begin{cases}
1/4, & \text{if $a=1$, $\gcd(a,b,c)=1$ and $d=1$;}\\
6-4\sqrt{2}\approx 0.343, & \text{if $a\geq 2$, $\gcd(a,b,c)=1$ and $d=1$;}\\
4/9 , & \text{if $\gcd(a,b,c)=1$ and $d\geq 3$ is odd;}\\
1/2, & \text{if $\gcd(a,b,c)>1$ or $a=b=c=1$ or $d\ge 2$ is even.}
\end{cases}
\end{equation}

Given positive integers $a\leq b \leq c$, let $K_{a,b,c}$ be the complete 3-partite 3-graph with three parts of size $a,b$, and $c$, respectively. 

\begin{theorem}[Main Result]\label{thm:main}
For integers $1\le a\le b\le c$,
\begin{equation*}
t_{1}(n, K_{a,b,c})= \left(\max\left\{f(a,b,c), 1-\left(\frac {b+c}{a+b+c} \right)^2, \left(\frac {a+b}{a+b+c} \right)^2 \right\} + o(1) \right)\binom n2.
\end{equation*}
\end{theorem}

Let us compare Theorem~\ref{thm:main} with the corresponding result in \cite{My14}, which states that
\begin{equation*}
t_2(n, K_{a, b, c}) =
\begin{cases}
n/2 + o(n) & \text{if $\gcd(a,b,c)>1$ or $a=b=c=1$;}  \\
an/(a+b+c) + o(n) & \text{if $\gcd(a,b,c)=1$ and $d=1$;} \\
\max \{ an/(a+b+c), n/p \} + o(n) & \text{otherwise.}
\end{cases}
\end{equation*}
where $p$ is the smallest prime factor of $d$. Not only is Theorem~\ref{thm:main} more complicated, but also it contains a case where the coefficient of the threshold is irrational. In fact, as far as we know, all the previously known tiling thresholds had rational coefficients.

The lower bound in Theorem \ref{thm:main} follows from six constructions given in Section 2. Three of them are
known as \emph{divisibility barriers} and two are known as \emph{space barriers}. Roughly speaking, the divisibility barriers, known as lattice-based constructions, only prevent the existence of a perfect $K$-tiling; in contrast, the space barriers are `robust' because they prevent the existence of an almost perfect $K$-tiling. Our last construction is based on the fact that if a perfect $K$-tiling exists, then every vertex is covered by a copy of $K$, so we call it a \emph{covering barrier}. Such a barrier has never appeared before -- see concluding remarks in Section 5.

Our proof of the upper bound of Theorem \ref{thm:main} consists of two parts: one is on finding an almost perfect $K$-tiling in $H$, and the other is on `finishing up' the perfect $K$-tiling.
Our first lemma says that $H$ contains an almost perfect $K$-tiling if the minimum vertex degree of $H$ exceeds those of the space barriers.

\begin{lemma}[Almost Tiling Lemma] \label{lem:alm_til}
Fix integers $1\le a\le b\le c$. For any $\r>0$ and $\a>0$, there exists an integer $n_0$ such that the following holds. Suppose $H$ is a $3$-graph of order $n>n_0$ with $\delta_1(H)\geq (\max\{1-(\frac {b+c}{a+b+c})^2,(\frac {a+b}{a+b+c})^2\}+\gamma)\binom n2$, then there exists a $K_{a,b,c}$-tiling covering all but at most $\a n$ vertices.
\end{lemma}

The absorbing method, initiated by R\"odl, Ruci\'nski and Szemer\'edi \cite{RRS06}, has been shown to be effective in finding spanning (hyper)graphs. Our absorbing lemma says that $H$ contains a small $K_{a,b,c}$-tiling that can \emph{absorb} any much smaller set of vertices of $H$ if the minimum vertex degree of $H$ exceeds those of the divisibility barriers and the covering barrier.

\begin{lemma}[Absorbing Lemma]\label{lem.absorbing}
Fix integers $1\le a\le b\le c$. For any $\r> 0$, there exists $\alpha>0$ such that the following holds for sufficiently large $n$. Suppose $H$ is a $3$-graph on $n$ vertices such that
\[
\delta_1(H)\ge (f(a,b,c)+\r)\binom n2.
\]
Then there exists a vertex set $W$ with $|W|\leq \frac14\gamma n$ such that for any vertex set $U\subset V(H)\setminus W$ with $|U|\leq \alpha n$ and $|U|\in (a+b+c)\mathbb{Z}$, both $H[W]$ and $H[W\cup U]$ have $K_{a,b,c}$-factors.
\end{lemma}

The upper bound of $t_{1}(n, K_{a,b,c})$ in Theorem \ref{thm:main} follows from Lemmas \ref{lem:alm_til} and \ref{lem.absorbing} easily.

\begin{proof}[Proof of Theorem \ref{thm:main} (upper bound)]
Let $1\le a\le b\le c$ be integers and $\r>0$.
Let $\a>0$ be the constant returned by Lemma~\ref{lem.absorbing} and let $n\in (a+b+c)\mathbb{N}$ be sufficiently large.
Suppose that $H$ is a 3-graph on $n$ vertices with $\delta_1(H)\ge (\delta+\r) \binom n2$, where
\[
\delta= \max\left\{f(a,b,c), 1-\left(\frac {b+c}{a+b+c} \right)^2, \left(\frac {a+b}{a+b+c} \right)^2 \right\}.
\]
We apply Lemma \ref{lem.absorbing} to $H$ and get a vertex set $W$ with $|W|\leq \frac14\gamma n$ and the described absorbing property.
In particular, $|W|\in (a+b+c)\mathbb{N}$.
Let $H'=H[V(H)\setminus W]$. Then
\[
\delta_1(H')\ge \delta_1(H) - |W| (n-1) \ge (\delta+\r) \binom n2 - \frac{\r}2 \binom n2 \ge \left(\delta+ \frac{\r}2\right)\binom {|V(H')|}2.
\]
Next we apply Lemma \ref{lem:alm_til} on $H'$ with $\r/2$ in place of $\r$ and get a $K_{a,b,c}$-tiling covering $T$ all but a set $U$ of at most $\a |V(H')|< \a n$ vertices of $H'$.
Since $|V(H)|, |W|, |V(T)|\in (a+b+c)\mathbb{N}$, $|U|=|V(H)|-|W|-|V(T)|\in (a+b+c)\mathbb{N}$. 
By the absorbing property of $W$, there exists a $K_{a,b,c}$-factor on $H[W\cup U]$.
Thus we get a $K_{a,b,c}$-factor of $H$.
\end{proof}

Although this proof is a straightforward application of the absorbing method, there are several new ideas in the proofs of Lemmas \ref{lem:alm_til} and \ref{lem.absorbing}. First, in order to show that almost every $(a+b+c)$-set has many absorbing sets, we use lattice-based absorbing arguments developed recently by Han \cite{Han14_poly}. Second, in order to prove Lemma \ref{lem:alm_til}, we use the concept of \emph{fractional homomorphic tiling} given by Bu\ss, H\`an and Schacht \cite{BHS}. Third, we need a recent result of F\"uredi and Zhao \cite{FuZh14} on the shadows of 3-graphs, which can be viewed as a vertex degree version of the well-known Kruskal-Katona Theorem for 3-graphs.

The rest of the paper is organized as follows.
We prove the lower bound in Theorem \ref{thm:main} by six constructions in Section 2. We prove Lemma \ref{lem.absorbing} in Section 3 and Lemma \ref{lem:alm_til} in Section 4, respectively. Finally, we give concluding remarks in Section 5.

\bigskip
\noindent\textbf{Notations.}
Throughout this paper we let $1\le a\le b\le c$ be three integers and $k=a+b+c\ge 3$.
When it is clear from the context, we write $K_{a,b,c}$ as $K$ for short.
By $x\ll y$ we mean that for any $y> 0$ there exists $x_0> 0$ such that for any $x< x_0$ the following statement holds. When $x\ll y\ll w$ and $x\ll z\ll w$, we simply write $x\ll \, y, z \, \ll w$ (and this should not be confused with $x\ll \, y$ and $z \, \ll w$). We omit the floor and ceiling functions when they do not affect the proof.

\section{Extremal examples}

In this section, we prove the lower bound in Theorem \ref{thm:main} by six constructions.
Following the definition in \cite{My14}, we say a 3-partite 3-graph $K_{a,b,c}$ is of \emph{type 0} if $\gcd(a,b,c)>1$ or $a=b=c=1$. We say $K_{a,b,c}$ is of \emph{type $d\ge 1$} if $\gcd(a,b,c)=1$ and $\gcd(b-a,c-b)=d$.

\begin{construction}[Space Barrier I] \label{cons:s1}
 Let $V_1$ and $V_2$ be two disjoint sets of vertices such that $|V_1|=\frac a{k}n-1$ and $|V_1|+ |V_2|=n$. Let $G_1$ be the $3$-graph on $V_1\cup V_2$ whose edge set consists of all triples $e$ such that $|e\cap V_1|\ge 1$. Then $\delta_1(G_1)=\binom {n-1}2-\binom {(1-\frac {a}{k})n}2 = (1-\left(1-a/k\right)^2)\binom n2 + o(n^2)$. Since $a\le b \le c$, we have $a\le k/3$ and $0< 1-\left(1- {a}/{k}\right)^2 \le 5/9$.
\end{construction}

We claim that $G_1$ has no perfect $K_{a,b,c}$-tiling.
Indeed, consider a copy $K'$ of $K_{a,b,c}$ in $G_1$. We observe that at least one color class of $K'$ is a subset of $V_1$ -- otherwise $V_2$ contains at least one vertex from each color class; since $K'$ is complete, there is an edge in $V_2$, contradicting the definition of $G_1$. Hence a $K_{a,b,c}$-tiling in $G_1$ covers at most $\frac{|V_1|}a k<n$ vertices, so it cannot be perfect.

\begin{construction}[Space Barrier II] \label{cons:s2}
Let $V_1$ and $V_2$ be two disjoint sets of vertices such that $|V_1|=\frac {a+b}{k}n-1$ and $|V_1|+ |V_2|=n$. Let $G_2$ be the $3$-graph on $V_1\cup V_2$ whose edge set consists of all triples $e$ such that $|e\cap V_1|\geq 2$. Then $\delta_1(G_2)=\binom {\frac {a+b}{k}n-1}2=((a+b)^2/k^2)\binom n2 + o(n^2)$. Since $a\le b \le c$, we have $a+b\le 2k/3$ and $0< (a+b)^2/k^2 \leq 4/9$.
\end{construction}

We claim that $G_2$ has no perfect $K_{a,b,c}$-tiling.
Similarly as in the previous case, for any copy $K'$ of $K_{a,b,c}$ in $G_2$, at least two color classes of $K'$ are subsets of $V_1$. Hence a $K_{a,b,c}$-tiling in $G_2$ covers at most $\frac{|V_1|}{a+b}k<n$ vertices, so it cannot be perfect.

\begin{construction}[Divisibility Barrier I] \label{cons:d1}
Let $V_1$ and $V_2$ be two disjoint sets of vertices such that $|V_1|= \lfloor\frac n2\rfloor+1$ and $|V_1|+ |V_2|=n$. Let $H_1$ be the $3$-graph on $V_1\cup V_2$ such that $H_1[V_1]$ and $H_1[V_2]$ are two complete $3$-graphs. Then $\delta_1(H_1)\ge \binom {\frac n2 -2}2= \frac14\binom n2 + o(n^2)$.
\end{construction}

We claim that $H_1$ has no perfect $K_{a,b,c}$-tiling.
Indeed, each copy of $K_{a,b,c}$ must be a subgraph of $H_1[V_1]$ or $H_1[V_2]$. Since $k\ge 3$, due to the choice of $V_1$ and $V_2$, we have $|V_1|\neq |V_2|$ mod $k$ and therefore $k$ cannot divide both $|V_1|$ and $|V_2|$. Hence $H_1$ has no perfect $K_{a,b,c}$-tiling.

\begin{construction}[Divisibility Barrier II] \label{cons:d2}
Suppose that $K_{a,b,c}$ is of type $d$ for some even $d$. Let $V_1$ and $V_2$ be two disjoint sets of vertices such that $|V_1|+ |V_2|=n$ and $|V_2|\in [\frac n2-2, \frac n2+2]$ is odd, and $\gcd(a,b,c)\nmid |V_2|$ if $\gcd(a,b,c)>1$. Note that we can pick $|V_2|$ satisfying these conditions because in the interval $[\frac n2-2, \frac n2+2]$, there are at least two consecutive odd numbers, therefore at least one of them is not divisible by $\gcd(a,b,c)$.
Let $H_2$ be the $3$-graph on $V_1\cup V_2$ whose edge set consists of all triples $e$ such that $|e\cap V_2|$ is even (0 or 2).
Then $\delta_1(H_2)=\min\{\binom{|V_1|-1}2+\binom{|V_2|}2, |V_1|(|V_2|-1)\}= \frac12\binom n2 + o(n^2)$.
\end{construction}

We claim that $H_2$ has no perfect $K_{a,b,c}$-tiling.
Consider a copy $K'$ of $K_{a,b,c}$ in $H_2$.
Since every edge intersects $V_2$ in an even number of vertices and $K'$ is complete, no color class of $K'$ intersects both $V_1$ and $V_2$. Moreover, either 0 or 2 color classes of $K'$ are subsets of $V_2$.
Thus $|V(K')\cap V_2|\in \{0, a+b, a+c, b+c\}$.
If $\gcd(a,b,c)>1$, then $|V(K')\cap V_2|$ is divisible by $\gcd(a,b,c)$.
Since $\gcd(a,b,c)\nmid |V_2|$, there is no perfect $K_{a,b,c}$-tiling. Otherwise, either $a=b=c=1$ or $\gcd(b-a, c-b)$ is even.
In either case, all of $a+b,a+c$ and $b+c$ are even and thus $|V(K')\cap V_2|$ is even.
Since $|V_2|$ is odd, $H_2$ has no perfect $K_{a,b,c}$-tiling.

\begin{construction}[Divisibility Barrier III] \label{cons:d3}
Suppose that $K_{a,b,c}$ is of type $d$ for some odd $d\ge 3$, let $V_1$ and $V_2$ be two disjoint sets of vertices such that $|V_1|+ |V_2|=n$ and $|V_1|\in [\frac n3-1, \frac n3+1]$ and $d \nmid (|V_1|-\frac{n}{k}a)$.
Let $H_3$ be the $3$-graph on $V_1\cup V_2$ whose edge set consists of all triples $e$ such that $|e\cap V_1|=1$. Then $\delta_1(H_3)=\min\{|V_1|(|V_2|-1), \binom {|V_2|}2\}= \frac49\binom n2 + o(n^2)$.
\end{construction}

We claim that $H_3$ has no perfect $K_{a,b,c}$-tiling.
Consider a copy $K'$ of $K_{a,b,c}$ in $H_3$. Similarly as in the previous case, exactly one color class of $K'$ is a subset of $V_1$, which implies $|V_1\cap V(K')|\in \{a, b, c\}$. Since $\gcd (b-a, c-b) = d$, we have $a\equiv b\equiv c$ mod $d$ and thus $|V_1\cap V(K')|\equiv a$ mod $d$. If $H_3$ contains a perfect $K_{a,b,c}$-tiling $\K$, then $|V_1| - \frac nk a=|V(\K)\cap V_1|-\frac{n}{k}a\equiv 0$ mod $d$, contradicting our assumption on $|V_1|$. Hence $H_3$ has no perfect $K_{a,b,c}$-tiling.

\begin{construction}[Covering Barrier] \label{cons:t}
Let $\a=\sqrt{2}-1$ and suppose that $V$ is partitioned into $\{v\}\cup V_1\cup V_2\cup V_3$ such that $|V_1|=|V_2|=\a n$ and $|V|=n$. Define a $3$-graph $F$ on $V$ whose edge set consists of all triples $vxy$ with $x\in V_1, y\in V_2$ and all triples $e$ in $V_1\cup V_2 \cup V_3$ such that $e\cap V_1=\emptyset$ or $e\cap V_2=\emptyset$. Therefore, $\delta_1(F)= (6-4\sqrt{2})\binom n2 + o(n^2)\approx 0.343\binom n2$.
\end{construction}

It is easy to see that $v$ is not contained in any copy of $K_{2,2,2}$, and hence not contained in any copy of $K_{a,b,c}$ with $a>1$. Therefore, $F$ has no perfect $K_{a,b,c}$-tiling with $a>1$.

\smallskip
\begin{proof}[Proof of Theorem \ref{thm:main} $($lower bound$)$]
Given positive integers $a\le b\le c$ and $n\in k\mathbb{N}$, where $k = a+ b+c$, let $t_1 = t_1(n, K_{a, b, c})$ be the tiling threshold.
By Constructions~\ref{cons:s1} and \ref{cons:s2}, we have $t_1\ge  (1- (1- {a}/{k})^2)\binom n2 + o(n^2)$ and $t_1\ge ((a+b)^2/k^2)\binom n2+o(n^2)$. 
Furthermore, assume $K_{a,b,c}$ has type $d$.
First, by definition, $d$ is even if and only if $\gcd(a,b,c)>1$, or $a=b=c=1$, or $d\ge 2$ is even. By Construction~\ref{cons:d2}, we have $t_1 \ge \frac12 \binom n2 + o(n^2)$ in this case.
Second, assume that $d\ge 3$ is odd, then by Construction~\ref{cons:d3}, we have $t_1 \ge \frac49\binom n2 + o(n^2)$.
Finally assume that $d=1$.
If $a=1$, by Construction~\ref{cons:d1}, we have $t_1 \ge \frac14\binom n2 + o(n^2)$. 
If $a\ge 2$, then by Construction~\ref{cons:t}, we have $t_1\ge (6-4\sqrt{2})\binom n2 + o(n^2)$. 
\end{proof}

\section{Proof of the Absorbing Lemma}

\subsection{Preparation}

We need a simple counting result, which, for example, follows from the result of Erd\H{o}s \cite{erdos} on supersaturation. Given $l_1, \dots, l_ r\in \mathbb{N}$, let $K^{(r)}_{l_1, \dots, l_ r}$ denote the complete $r$-partite $r$-graph whose $j$th part has exactly $l_j$ vertices for all $j\in [r]$.

\begin{proposition}\label{supersaturation}
Given $\mu> 0$, $r, m, l_1, \dots, l_ r\in \mathbb{N}$, there exists $\mu'>0$ such that the following holds for sufficiently large $n$. Let $H$ be an $r$-graph on $n$ vertices with a vertex partition $V_1 \cup \dots \cup V_m$. Suppose $i_1, \dots, i_r\in [m]$ and $H$ contains at least $\mu {n}^{r}$ edges $e=\{ v_1, \dots, v_r \}$ such that $v_1\in V_{i_1}$, $\dots, v_r\in V_{i_r}$. Then $H$ contains at least $\mu' {n}^{l_1+\cdots + l_r}$ copies of $K^{(r)}_{l_1, \dots, l_ r}$ whose $j$th part is contained in $V_{i_j}$ for all $j\in [r]$.
\end{proposition}

Given a $3$-graph $H$, its \emph{shadow} $\partial H$ is the set of the pairs that are contained in at least one edge of $H$. We need a recent result of F\"uredi and Zhao \cite{FuZh14} on the shadows of 3-graphs. The union of two (overlapping) complete 3-graphs of order about $\sqrt{d} n$ shows that Lemma~\ref{shadow} is (asymptotically) best possible.

\begin{lemma}\cite{FuZh14}\label{shadow}
Given $d\in [\frac14, \frac{47 - 5\sqrt{57}}{24})$, let $n$ be sufficiently large. 
If $H$ is a $3$-graph on $n$ vertices with $\delta_1(H)\ge d\binom{n}2$, then $|\partial H|\geq (4\sqrt{d}-2d-1)\binom n2$.
\end{lemma}

The next lemma says that for any 3-graph, after a removal of a small portion of edges, any two vertices with a positive codegree in the remaining 3-graph has a linear codegree in $H$.

\begin{lemma}\label{lem:He}
Given $\e>0$ and an $n$-vertex 3-graph $H=(V, E)$, there exists a vertex set $V_0'\subseteq V$ and a subhypergraph $H'$ of $H$ such that the following holds
\begin{itemize}
\item[(i)] $|V_0'|\le 3\e n$,
\item[(ii)] $\deg_{H'}(v)\ge \deg_H(v) - \e\binom n2$ for any $v\in V\setminus V_0'$,
\item[(iii)] $\deg_H(S)> \e^2 n$ for any pair of vertices $S\in \partial H'$.
\end{itemize}
\end{lemma}

\begin{proof}
If an edge $e\in E(H)$ contains a pair $S\in \binom e2$ with $\deg_H(S)\leq \e^2 n$, then it is called \emph{weak}, otherwise called \emph{strong}. Let $H'$ be the subhypergraph of $H$ induced on strong edges. Then (iii) holds. 
Let
\begin{align*}
V_{0}'=\left\{v\in V: v \text{ is contained in at least } \e \binom{n}{2} \text{ weak edges}\right\}.
\end{align*}
Then (ii) holds.  Note that the number of weak edges in $H$ is at most $\binom n2\e^2n$. If $|V_{0}'|> 3\e n$, then there are more than $3\e n\cdot \e \binom{n}{2}/3=\binom n2 \e^2n$ weak edges in $H$, a contradiction. Thus (i) holds.
\end{proof}

We use the notion of reachability introduced by Lo and Markstr\"om \cite{LM2, LM1}.
Given an $r$-graph $F$ of order $f$, $\b>0$, $i \in \mathbb{N}$, two vertices $u, v$ in an $r$-graph $H$ on $n$ vertices are \emph{$(F, \b ,i)$-reachable (in $H$)} if and only if there are at least $\b n^{if-1}$ $(if-1)$-sets $W$ such that both $H[\{u\} \cup W]$ and $H[\{v\} \cup W]$ contain $F$-factors. In this case, we call $W$ a \emph{reachable set} for $u$ and $v$. A vertex set $A$ is \emph{$(F, \b ,i)$-closed in $H$} if every two vertices in $A$ are $(F, \b ,i)$-reachable in $H$.
For $x\in V(H)$, let $\tilde{N}_{F, \beta, i}(x)$ be the set of vertices that are $(F, \beta, i)$-reachable to $x$.

We use the following two results from \cite{LM1}.

\begin{proposition}\cite[Proposition 2.1]{LM1}\label{reachablesteps}
Given $\beta, \e>0$ and positive integers $f$ and $i_0'>i_0$, there exists $\beta'>0$ such that the following holds for sufficiently large $n$.
Let $F$ be an $r$-graph on $f$ vertices.
Given an $n$-vertex $r$-graph $H$ and a vertex $x\in V(H)$ with $|\tilde{N}_{F,\beta, i_0}(x)|\ge \e n$, then $\tilde{N}_{F,\beta, i_0}(x)\subseteq \tilde{N}_{F,\beta', i_0'}(x)$.
In other words, if $x,y\in V(H)$ are $(F,\b,i_0)$-reachable in $H$ and $|\tilde{N}_{F,\beta, i_0}(x)|\ge \e n$, then $x,y$ are $(F,\b',i_0')$-reachable in $H$.
\end{proposition}

The following lemma is essentially \cite[Lemma 4.2]{LM1}. In fact, \cite[Lemma 4.2]{LM1} shows that the density of $K_{c,c,c+1}$'s containing both $x$ and $y$ in the part of size $c+1$ is positive. By averaging, this implies that the density of $K_{a,b,c+1}$'s containing both $x$ and $y$ in the part of size $c+1$ is positive.

\begin{lemma}\cite{LM1}\label{agood}
Let $a\le b\le c$ be positive integers and $K=K_{a,b,c}$.
Given $\e>0$, there exists $\eta>0$ such that the following holds for sufficiently large $n$. 
For any $n$-vertex $3$-graph $H$, two vertices $x, y \in V(H)$ are $(K,\eta,1)$-reachable if the number of pairs $S\in N(x)\cap N(y)$ with $\deg(S)\geq \e n$ is at least $\e \binom n2$.
\end{lemma}

\subsection{Auxiliary Lemmas}
Given positive integers $a\le b\le c$, let $K=K_{a,b,c}$, and $k=a+b+c\ge 3$.
We call an $m$-set $A$ an \emph{absorbing $m$-set} for a $k$-set $S$ if $A\cap S=\emptyset$ and both $H[A]$ and $H[A\cup S]$ contain $K$-factors. Denote by $\A^m(S)$ the set of all absorbing $m$-sets for $S$.

Our proof of the Absorbing Lemma is based on the following lemma.

\begin{lemma}\label{lemabsorbing}

Given $0<\eta \le 1/(2k)$, $\b>0$, and $i_0 \in \mathbb{N}$, there exists $\alpha>0$ such that the following holds for all sufficiently large integers $n$. Suppose $H= (V, E)$ is an $n$-vertex $3$-graph with the following two properties
\begin{itemize}
\item[($\diamondsuit$)] For any $v\in V$, there are at least $\eta n^{k-1}$ copies of $K$ containing it.\label{condition2}
\item[($\triangle$)] There exists $V_0\subset V$ with $|V_0|\leq \eta^2 n$ such that $V(H)\setminus V_0$ is $(K, \b ,i_0)$-closed in $H$.\label{condition1}
\end{itemize}
Then there exists a vertex set $W$ with $V_0\subseteq W\subseteq V$ and $|W|\leq \eta n$ such that for any vertex set $U\subseteq V\setminus W$ with $|U|\leq \alpha n$ and $|U|\in k\mathbb{Z}$, both $H[W]$ and $H[U\cup W]$ contain $K$-factors.
\end{lemma}

\begin{proof}
Let
\[
\eta_1=\frac{\eta}2 \left(\frac{\b}2 \right)^{k-1}\text{ and }\a = \frac{\eta_1^2}{32 i_0 k}.
\]
There are two steps in our proof. In the first step, we build an absorbing family $\F_1$ that can absorb any small portion of vertices in $V\setminus V_0$. In the second step, we put the vertices in $V_0\setminus V(\F_1)$ into a family $\F_2$ of copies of $K$. Then $V(\F_1\cup\F_2)$ is the desired absorbing set.

Fix a $k$-set $S=\{v_1,v_2,\dots,v_k\}\subset V\setminus V_0$. Let $m=i_0k^2-i_0k$. We claim that there are at least $\eta_1 n^m$ absorbing $m$-sets for $S$, namely, $|\A^m(S)|\geq \eta_1 n^m$.
Indeed, we first find a $k$-set $S'=\{v_1, u_2,\dots, u_k\}\subset V\setminus V_0$ such that $S'\cap S=\{v_1\}$ and $S'$ spans a copy of $K$. By $(\diamondsuit)$, there are at least
\[
\eta n^{k-1}-(k-1)n^{k-2} - (\eta^2 n) n^{k-2}\ge  {\eta} n^{k-1}/2
\]
choices for $S'$ because there are at most $n^{k-2}$ $k$-sets containing $v_1$ and another fixed vertex, and $|V\setminus V_0| \le \eta^2 n$.
For each $i=2,\dots,k$, since $V\setminus V_0$ is $(K,\b,i_0)$-closed, there are at least $\b n^{i_0k-1}$ reachable $(i_0 k-1)$-sets $S_i$ for $u_i$ and $v_i$. 
We greedily choose pairwise disjoint sets $S_2, \dots, S_k$ -- when choosing $S_i$, we need to avoid the vertices in $S \cup S' \cup \bigcup_{j=2}^{i-1} S_j$ so there are at least $\b n^{i_0k-1}/2$ choices for $S_i$. 
Let $A=(S'\setminus \{v_1\})\cup \left(\bigcup_{i=2}^k S_i\right)$, then $|A|=m$.
We claim that both $H[A]$ and $H[A\cup S]$ contain $K$-factors.
Indeed, by the definition of reachability, each $S_i\cup \{u_i\}$ spans $i_0$ copies of $K$ and thus $H[A]$ contains a $K$-factor. Furthermore, since $S'$ spans a copy of $K$ and each $S_i\cup \{v_i\}$ spans $i_0$ copies of $K$, $H[A\cup S]$ also contains a $K$-factor.
Thus $A$ is an absorbing $m$-set for $S$. In total, we get at least
\[
 \frac{\eta}2 n^{k-1}\left(\frac{\b}2 n^{i_0k-1}\right)^{k-1} = \eta_1 n^m
 \]
such $m$-sets, thus $|\A^m(S)|\geq \eta_1 n^m$.

Now we build the family $\F_1$ by standard probabilistic arguments. Choose a family $\F$ of $m$-sets in $H$ by selecting each of the $\binom nm$ possible $m$-sets independently with probability $p= \eta_1 n^{1-m}/(8m)$. Then by Chernoff's bound, with probability $1-o(1)$ as $n \rightarrow \infty$, the family $\F$ satisfies the following properties:
\begin{align}
|\F|\leq 2p\binom nm\leq \frac{\eta_1 n}{4m}\, \text{ and }\, |\A^m(S)\cap \F|\geq \frac{p|\A^m(S)|}2\geq \frac{\eta_1^{2} n}{16m}, \text{ for all }S\in \binom {V\setminus V_0}k.  \label{expected}
\end{align}
Furthermore, the expected number of pairs of $m$-sets in $\F$ that are intersecting is at most
\begin{align*}
\binom nm\cdot m \cdot \binom n{m-1} \cdot p^2\leq \frac{\eta_1^2n}{64m}.
\end{align*}
Thus, by using Markov's inequality, we derive that with probability at least $1/2$,
\begin{align}
\F \text{ contains at most } \frac{\eta_1 ^{2}n}{32m} \text{ intersecting pairs of $m$-sets.}  \label{intersecting}
\end{align}
Hence, there exists a family $\F$ with the properties in \eqref{expected} and \eqref{intersecting}. By deleting one member of each intersecting pair and removing $m$-sets that are not absorbing sets for any $k$-set $S\subseteq V\setminus V_0$, we get a subfamily $\F_1$ consisting of pairwise disjoint $m$-sets.
Let $W_1=V(\F_1)$ and thus $|W_1|=|V(\F_1)|=m|\F_1|\leq m|\F|\leq \eta_1 n/4$.
Since every $m$-set in $\F_1$ is an absorbing $m$-set for some $k$-set $S$, $H[W_1]$ has a $K$-factor.
For any $k$-set $S$, by \eqref{expected} and \eqref{intersecting} above we have
\begin{align}
|\A^m(S)\cap \F_1|\geq \frac{\eta_1^{2} n}{16m}-\frac{\eta_1^2 n}{32m}=\frac{\eta_1^{2} n}{32m}.
\end{align}
For any set $U\subseteq V\setminus (V_0\cup W_1)$ of size $|U|\leq \a n$ and $|U|\in k\mathbb{Z}$, we arbitrarily partition it into at most $\frac{\a n}k$ $k$-sets. 
By the definition of $\F_1$, each such $k$-set has at least $\frac{\eta_1^2 n}{32m}\ge \frac{\a n}k$ absorbing sets in $\F_1$ so we can find a distinct absorbing set in $\F_1$ for each of the $k$-sets. As a result, $H[U\cup W_1]$ contains a $K$-factor.

In the second step, by ($\diamondsuit$), we greedily build $\F_2$, a collection of copies of $K$  that cover the vertices in $V_0\setminus W_1$. Indeed, assume that we have built $i<|V_0\setminus W_1|\leq \eta^2 n$ copies of $K$. Together with the vertices in $W_1$, at most $ki+\eta_1 n/4\leq k\eta^2 n+\eta_1 n/4$ vertices have already been covered by $\F$. So for any vertex $v\in V_0$ not yet covered, we find the desired copy of $K$ containing $v$ by ($\diamondsuit$), because $(k\eta^2 n+\eta_1 n/4)\cdot n^{k-2}< \eta n^{k-1}$.

Let $W=V(\F_2)\cup W_1$, we get the desired absorbing set $W$ with $|W|\le k\eta^2 n+\eta_1 n/4<\eta n$.
\end{proof}

So it remains to show that $(\diamondsuit)$ and $(\triangle)$ hold in the 3-graph $H$.
We first study the property $(\diamondsuit)$.
Throughout this subsection, let $d_0=6-4\sqrt{2}\approx 0.343$. Note that $(4\sqrt{d_0}-2d_0-1)+d_0=1$ because $\sqrt{d_0}=2-\sqrt{2}$.

\begin{lemma} \label{erdos}
For any $\gamma>0$, there exists $\eta>0$ such that the following holds for sufficiently large $n$. Let $H$ be an $n$-vertex $3$-graph with $\delta_1(H)\geq (d_0+\gamma)\binom n2$. Then each vertex $v\in V(H)$ is contained in at least $\eta n^{k-1}$ copies of $K$.
\end{lemma}
\begin{proof}
Let $\e=\r/12$. 
Let $\eta$ be returned by Lemma \ref{agood} when $\r\e^2/2$ plays the role of $\e$. Suppose that $n$ is sufficiently large and $H$ is an $n$-vertex $3$-graph with $\delta_1(H)\geq (d_0+\gamma)\binom n2$.
We apply Lemma \ref{lem:He} on $H$ and get $V_0'$ and $H'$ satisfying (i) -- (iii). Let $H'' = H'[V\setminus V_0']$ and $n'=|V\setminus V_0'|$. By Lemma~\ref{lem:He} (ii), we have
\[
\delta_1(H'')\geq \left(d_0+\gamma\right)\binom n2 - \e \binom{n}2 - |V_0'| (n-2)>d_0\binom {n'}2
\]
because $|V_0'|\le 3\e n$.
Since $ \frac14 < d_0 < \frac{47 - 5\sqrt{57}}{24}\approx 0.385$ and $n'\ge (1-3\e)n$ is sufficiently large, Lemma \ref{shadow} implies that
\[
\partial H'' \ge (4\sqrt{d_0}-2d_0-1)\binom{n'}2\ge (4\sqrt{d_0}-2d_0-1)(1-6\e)\binom n2.
\]
Since $\delta_1(H)\ge (d_0+\r)\binom n2$, for every $x\in V(H)$, we have
\[
|N_H(x)\cap \partial H''|\ge \left( d_0 + \r + (4\sqrt{d_0}-2d_0-1) - 6\e - 1\right) \binom{n}2 \ge \frac{\r}2 \binom n2,
\]
by the definitions of $d_0$ and $\e$.

Fix $x\in V(H)$ and note that every $S\in N_H(x)\cap \partial H''$ has degree at least $\e^2 n$ in $H$. Therefore, the number of $(S,y)$ with $S\in N_H(x)\cap \partial H''$ and $y\in N_{H}(S)$ is at least $\frac{\r}2 \binom n2\cdot \e^2 n$. By averaging, there exists a vertex $y$ such that
\[
|N_H(y)\cap N_H(x)\cap \partial H''|\geq \gamma \e^2 \binom n2/2.
\]
This means that $x$ and $y$ have at least $\gamma\e^2\binom n2/2$ common neighbors with degree at least $\e^2n$. By Lemma \ref{agood}, $x$ and $y$ are $(K,\eta,1)$-reachable. Hence, there are at least $\eta n^{k-1}$ $(k-1)$-sets $W$ such that $H[\{x\}\cup W]$ forms a copy of $K$.
\end{proof}

Now we consider the property $(\triangle)$. Following the approach in \cite{Han14_poly}, given a 3-graph $H$, we first find a partition of $V(H)$ such that all but one part are $(K,\beta, i)$-closed in $H$ and then study the reachability between different parts.
The following lemma provides such a partition.

\begin{lemma}\label{lem:P}
Given $\delta\ge 1/4$ and $\r>0$, there exist constants $0<\beta \ll \e \ll \r$ such that the following holds for sufficiently large $n$. Let $H$ be an $n$-vertex $3$-graph with $\delta_1(H)\ge (\delta+\r)\binom{n}{2}$. Then there is a partition $\cP$ of $V(H)$ into $V_0, V_1,\dots, V_t$ such that
\begin{itemize}
\item $|V_0|\le 4\e n$,
\item $t \le \lfloor 1/(\delta+\r/2) \rfloor$, and
\item $|V_i|\ge \e^2 n$ and $V_i$ is $(K,\beta, 2^{\lfloor 1/(\delta+\r/2) \rfloor-1})$-closed in $H$ for all $i\in [t]$.
\end{itemize}
\end{lemma}

\begin{proof}
Let $s=\lfloor 1/(\delta+\r/2) \rfloor$. Then we may choose $\e>0$ such that $\e\ll \min \{ (s+1)(\delta+\r/2) - 1, 1/k\}$.
Let $\eta$ be the constant returned from applying Lemma \ref{agood} with $\e^2/16$ in place of $\e$.
Note that we may require $\eta \ll \e$ because the conclusion of Lemma~\ref{agood} holds with $\eta$ replaced by any positive $\eta'<\eta$.
Furthermore, let
\[
1/n\ll \beta = \beta_{s-1} \ll \cdots \ll \beta_1 \ll \beta_0 \le \,\eta, \a\, \ll \e. 
\]
Let $H=(V, E)$ be an $n$-vertex 3-graph with $\delta_1(H)\ge (\delta+\r)\binom{n}{2}$.
We apply Lemma \ref{lem:He} to $H$ and obtain $V_0'$ and $H'$ satisfying (i) -- (iii). 

Given $v\in V$ and $0\le i\le s-1$, let $\tilde N_{i}(v) =\tilde{N}_{K,\beta_i, 2^i}(v)$ be the set of vertices in $H$ that are $(K,\beta_i, 2^i)$-reachable to $x$ (note that $\tilde N_{i}(v)$ may contain the vertices of $V'_0$).
Throughout this proof, we say $2^i$-reachable (respectively, $2^i$-closed) for $(K,\beta_i, 2^i)$-reachable (respectively, $(K,\beta_i, 2^i)$-closed) for short.

Fix $x\in V\setminus V_0'$, we claim that $|\tilde N_0(x)|\ge \frac34 \e^2 n$. To see this, let
\[
D=\left\{v\in V: |N_{H'}(v)\cap N_{H'}(x)|\geq \frac{\e^2}{16} \binom n2\right\}.
\]
Since $\deg_H(p)> \e^2 n$ for any $p\in \partial H'$, Lemma \ref{agood} implies that two vertices $x, v \in V$ are $1$-reachable if $|N_{H'}(v)\cap N_{H'}(x)| \ge \e^2 \binom n2/16$. 
Therefore $D\subseteq \tilde N_0(x)$. 
Let $t$ be the number of pairs $(p, u)$ where $p\in N_{H'}(x)$ and $u\in N_{H'}(p)$. By Lemma~\ref{lem:He} (iii), we have $t\geq \deg_{H'}(x) \cdot \e^2 n$.
We also know that $|N_{H'}(v)\cap N_{H'}(x)|< \frac{\e^2}{16} \binom n2$ if $v\notin D$, and $|N_{H'}(v)\cap N_{H'}(x)|\le \deg_{H'}(x)$ otherwise. Consequently,
\[
\deg_{H'}(x) \, \e^2 n\leq t \leq n\cdot \frac{\e^2}{16}\binom n2+|D|\cdot \deg_{H'}(x).
\]
So we get $|D|\ge \e^2 n  -  \e^2 n \binom n2/ (16\deg_{H'}(x))$. Since $x\in V\setminus V_0'$ and $\delta\ge 1/4$, by Lemma~\ref{lem:He} (ii), 
\begin{equation}\label{eq:dHe}
\deg_{H'}(x) \ge (\delta + \r - \e)\binom n2 \ge \left(\delta + \frac{\r}2 \right)\binom n2 > \frac14\binom n2.
\end{equation}
Consequently $ |D|\ge \frac34 \e^2n$ and in turn, $|\tilde N_0(x)|\ge \frac34 \e^2 n$.

Since $|\tilde N_0(x)| \ge \frac34 \e^2 n$, by Proposition \ref{reachablesteps} and the choice of $\beta_i$'s, we derive $\tilde {N}_{i}(x)\subseteq \tilde {N}_{i+1}(x)$ for all $0\le i<s-1$ and all $x\in V\setminus V_0'$. Furthermore, 
if a set $W\subseteq V\setminus V_0'$ is $2^i$-closed in $H$ for some $i\le s-1$, then $W$ is $ 2^{s-1}$-closed in $H$.

Given any set $S\subseteq V\setminus V_0'$ of $s+1$ vertices, by the Inclusion-Exclusion principle, 
\[
\sum_{x\in S} \deg_{H'}(x) - \sum_{x, y\in S} | N_{H'}(x) \cap N_{H'}(y) | \le | \bigcup_{x\in S} N_{H'}(x) | \le \binom{n}{2},
\]
which implies $\sum_{x, y\in S} | N_{H'}(x) \cap N_{H'}(y) | \ge ( (s+1)(\delta+\r/2) - 1 ) \binom n2$ by \eqref{eq:dHe}. Since $\e\ll (s+1)(\delta+\r/2) - 1$, there are two vertices $x, y\in S$ such that $|N_{H'}(x)\cap N_{H'}(y)| \ge \frac{\e^2}{16} \binom n2$, so $x, y$ are 1-reachable to each other.
Consequently, if $s=1$, then $V\setminus V_0'$ is 1-closed and we get the desired partition $\cP=\{V_0', V\setminus V_0'\}$.

We may thus assume that $s\ge 2$ and there are two vertices in $V\setminus V_0'$ that are \emph{not} $2^{s-1}$-reachable to each other (otherwise we are done). Let $t'$ be the largest integer such that
there exist $v_1,\dots, v_{t'}\in V\setminus V_0'$ such that no two of them are $2^{s+1-t'}$-reachable to each other.
Earlier arguments show that $t'$ exists and $2\le t'\le s$.
Fix such $v_1,\dots, v_{t'}\in V\setminus V_0'$. By Proposition~\ref{reachablesteps}, we can assume that any two of them are not $2^{s-t'}$-reachable to each other. Then $\tilde N_{s-t'}(v_i)$, $i\in [t']$ satisfy the following properties.
\begin{enumerate}[(a)]
\item Any $v\in (V\setminus V_0')\setminus\{v_1,\dots, v_{t'}\}$ must be in $\tilde N_{s-t'}(v_i)$ for some $i\in [t']$ -- otherwise $\{v, v_1,\dots, v_{t'}\}$ contradicts the definition of $t'$.

\item $|\tilde N_{s-t'}(v_i)\cap \tilde N_{s-t'}(v_j)|<\a n$ for any $i\ne j$ -- otherwise there are at least
\begin{align*}
\frac{\a n}{(2^{s+1-t'}k-1)!} \left( \beta_{s-t'}n^{2^{s-t'}k-1} - n^{2^{s-t'}k-2} \right) \left(\beta_{s-t'}n^{2^{s-t'}k-1} - 2^{s-t'}k  \, n^{2^{s-t'}k-2} \right) 
\end{align*}
reachable $(2^{s+1-t'}k-1)$-sets for $v_i, v_j$ because there are at least $\a n$ vertices 
$w\in \tilde N_{s-t'}(v_i)\cap \tilde N_{s-t'}(v_j)$, at least $ \beta_{s-t'}n^{2^{s-t'}k-1} - n^{2^{s-t'}k-2}\ $ $2^{s-t'}$-reachable sets $S$ for $v_i$ and $w$ that do not contain $v_j$, and at least $\beta_{s-t'}n^{2^{s-t'}k-1} - 2^{s-t'}k n^{2^{s-t'}k-2}\ $ $2^{s-t'}$-reachable sets for $v_j$ and $w$ that avoid $\{v_i\}\cup S$; finally, we divide by $(2^{s+1-t'}k-1)!$ to eliminate the effect of over-counting. Since $\beta_{s-t'}\gg \beta_{s+1-t'}$, this gives at least $\beta_{s+1-t'} n^{2^{s+1-t'}k-1}$ reachable $(2^{s+1-t'}k-1)$-sets for $v_i, v_j$, contradicting the assumption that $v_i, v_j$ are not $2^{s+1-t'}$-reachable to each other.
\end{enumerate}

For $i\in [t']$, let $V_i=(\tilde N_{s-t'}(v_i)\cup \{v_i\})\setminus (V_0' \cup \bigcup_{j\in [t']\setminus \{i\}} \tilde N_{s-t'}(v_j))$. We observe that $V_i$ is $2^{s-t'}$-closed for all $i\in [t']$. Indeed, if there exist $u_1, u_2\in V_i$ that are not $2^{s-t'}$-reachable to each other, then $\{u_1, u_2, v_1,\dots, v_{t'} \} \setminus\{v_{i}\}$ contradicts the definition of $t'$. 
Without loss of generality, we may assume $|V_1|\ge \cdots \ge |V_{t'}|$.
Let $t$ be the largest integer $i\in [t']$ such that $|V_i| \ge \e^2 n$. 
Let $V_0=V\setminus (\bigcup_{1\le i\le t}V_i)$. 
Clearly $V_0'\subseteq V_0$. By (a) and (b), we have $|V_0|\le |V_0'| +\binom{t'}{2}\a n + t'\e^2 n\le 4\e n$. 
So $\cP=\{V_0, V_1, \dots, V_t\}$ is the desired partition.
\end{proof}

\medskip
We use the following definitions introduced by Keevash and Mycroft \cite{KM1}.
Let $r, t>0$ be integers and let $F$ be an $r$-graph of order $f$.
Suppose that $H$ is an $r$-graph with a partition $\cP=\{V_0, V_1,\dots, V_t\}$ of $V(H)$.
The \emph{index vector} $\mathbf{i}_{\cP}(S)\in \mathbb{Z}^t$ of a subset $S\subset V(H)$ with respect to $\cP$ is the vector whose coordinates are the sizes of the intersections of $S$ with $V_1,\dots, V_t$.
We call a vector $\mathbf{i}\in \mathbb{Z}^t$ an \emph{$s$-vector} if all its coordinates are nonnegative and their sum is $s$.
Given $\mu>0$, an $r$-vector $\mathbf{v}\in \mathbb{Z}^t$ is called a $\mu$\emph{-robust edge vector} if at least $\mu |V(H)|^r$ edges $e\in E(H)$ satisfy $\mathbf{i}_\cP(e)=\mathbf{v}$;
an $f$-vector $\mathbf{v}\in \mathbb{Z}^t$ is called a $\mu$\emph{-robust $F$-vector} if at least $\mu |V(H)|^f$ copies $F'$ of $F$ in $H$ satisfy $\mathbf{i}_\cP(V(F'))=\mathbf{v}$.
Let $I_{\cP}^{\mu}(H)$ be the set of all $\mu$-robust edge vectors and let $I_{\cP, F}^{\mu}(H)$ be the set of all $\mu$-robust $F$-vectors.
Let $L_{\cP, F}^{\mu}(H)$ be the lattice generated by the vectors of $I_{\cP,F}^{\mu}(H)$, in other words, $L_{\cP, F}^{\mu}(H)$ consists of all linear combinations of the vectors of $I_{\cP,F}^{\mu}(H)$.

When $\mathbf{i}$ is a $\mu$-robust edge vector and $F$ is a complete $r$-partite $r$-graph,
Proposition \ref{supersaturation} implies that there exists $\mu'>0$ such that the edges with index vector $\mathbf{i}$ give rise to at least $\mu'n^f$ copies of $F$ with certain index vectors.
For example, when $t=2$, $F=K_{a,b,c}$ (so $r=3$) and $(1,2)\in I_{\cP}^{\mu}(H)$, we have $(a,b+c), (b,a+c)$ and $(c,a+b)\in I_{\cP, F}^{\mu'}(H)$ for some $\mu'>0$.
For $j\in [t]$, let $\mathbf{u}_j\in \mathbb{Z}^t$ be the $j$th \emph{unit vector}, namely, $\mathbf{u}_j$ has 1 on the $j$th coordinate and 0 on other coordinates.

Given a partition $\cP=\{V_0, V_1, \dots, V_t\}$ of $V(H)$, the following lemma shows that $V(H)\setminus V_0$ is closed if $\bfu_j - \bfu_l\in L_{\cP, F}^{\mu'}(H)$ for all $1\leq j< l\leq t$. Because of potential applications to other problems, we prove the following lemma for $r$-graphs with $r\ge 3$.

\begin{lemma}\label{lattticec}
Let $i_0, r, t>0$ be integers and let $F$ be an $r$-graph of order $f$.
Given constants $\e, \beta, \mu'>0$, there exists $\beta'>0$ and an integer $i_0'>0$ such that the following holds for sufficiently large $n$.
Let $H$ be an $r$-graph on $n$ vertices with a partition $\cP=\{V_0, V_1, \dots, V_t\}$ such that for each $j\in [t]$, $|V_j|\ge \e^2 n$ and $V_j$ is $(F,\b,i_0)$-closed in $H$. 
If $\bfu_j - \bfu_l\in L_{\cP, F}^{\mu'}(H)$ for all $1\leq j< l\leq t$, then $V(H)\setminus V_0$ is $(F,\b',i_0')$-closed in $H$.
\end{lemma}

\begin{proof}
We call a set $I$ of $f$-vectors in $\mathbb{Z}^t$ \emph{a base} if all $\bfu_j - \bfu_l$, $j< l$, can be written as linear combinations of the vectors in $I$, namely, there exist $a^{j, l}_{\bfv}\in \mathbb{Z}$ such that $\bfu_j - \bfu_l = \sum_{\bfv\in I}a^{j, l}_{\bfv}\bfv$.
For example, the set of all $f$-vectors in $\mathbb{Z}^t$ is a base. 
Since there are $\binom{f+t-1}{t-1}$ $f$-vectors in $\mathbb{Z}^t$, there are at most $2^{\binom{f+t-1}{t-1}}$ bases.
Given a base $I$, we denote by $C_I$ the largest $|a^{j, l}_{\bfv}|$ over all $\bfv\in I$ and $j< l$. Let $C' = \max C_I$ over all bases $I$. 

Given integers $r, t, i_0$ and constants $\e, \beta, \mu'>0$, let $n$ be sufficiently large, in particular, $n\gg C'$. Suppose that $H$ is an $r$-graph satisfying all the assumptions, in particular, $\bfu_j - \bfu_l\in L_{\cP, F}^{\mu'}(H)$ for all $j< l$.
We claim that for any $j< l$, any $x_j\in V_j$ and any $x_l\in V_l$ are $(F,\beta_{j, l}, {i}_{j, l})$-reachable for some ${\b}_{j, l}>0$ and some ${i}_{j, l}\ge i_0$.
Once this is done, since $|\tilde{N}_{F,\beta, i_0}(v)|\ge |V_j|-1 \ge \e^2 n /2$ for any $j\in [t]$ and $v\in V_j$, we can apply Proposition \ref{reachablesteps} with $\e^2/2$ in place of $\e$ and $i_0'= \max\{i_{j, l}\}$ and derive that any $x_j\in V_j$ and any $x_l\in V_l$ are $(F,\tilde{\beta}, i'_0)$-reachable for some $\tilde{\beta} >0$. 
For the same reason, any two vertices in $V_j$, $j\in [t]$, are $(F,\b'', i'_0)$-reachable for some $\b''>0$. 
We thus conclude that any two vertices of $V(H) \setminus V_0$ are $(F,\b', i'_0)$-reachable with $\b' =\min\{\tilde{\b}, \b''\}$.

Below we prove this claim for $j=1$ and $l=2$. Let $I= I_{\cP, F}^{\mu'}(H)$. 
By our assumption,  there exist $a_{\bfv}\in \mathbb{Z}$, $\mathbf{v}\in I$ such that $\bfu_1 - \bfu_2=\sum_{\mathbf{v}} a_{\bfv}\mathbf{v}$ and $|a_{\bfv}|\le C'$ for all $\mathbf{v}\in I$.
For each $\mathbf{v}\in I$, if $a_\bfv\ge 0$, then let $p_{\bfv}=a_\bfv$ and $q_\bfv=0$;  otherwise let $p_\bfv=0$ and $q_\bfv=-a_\bfv$. Hence
\begin{align}
\bfu_1 - \bfu_2=\sum_{\mathbf{v}\in I} (p_{\bfv} - q_\bfv) \mathbf{v}
\quad i.e., \quad
\sum_{\mathbf{v}\in I}{q_{\mathbf{v}}}\mathbf{v} + \mathbf{u}_1=\sum_{\mathbf{v}\in I} {p_{\mathbf{v}}}\mathbf{v} + \bfu_2. \label{regroup11}
\end{align}
By comparing the sums of all the coordinates from two sides of either equation in \eqref{regroup11}, we obtain that
$
\sum_{\mathbf{v}\in I}{p_{\mathbf{v}}}=\sum_{\mathbf{v}\in I}{q_{\mathbf{v}}},
$
which  we denote by $C$. Then $C\le |I| C' \le \binom{f+t-1}{t-1}C' < \mu' n/(4f)$ because $n$ is sufficiently large.
We greedily select $p_{\mathbf{v}} + q_{\mathbf{v}}$ vertex-disjoint copies of $F$ with index vector $\mathbf{v}$ that do not contain $x_1$ or $x_2$ for all $\mathbf{v}\in I$. This gives rise to two disjoint families $\K^p$ and $\K^q$, where $\K^p$ consists of $p_{\mathbf{v}}$ copies of $F$ with index vector $\bfv$ for all $\bfv\in I$, and $\K^q$ consists of $q_{\mathbf{v}}$ vertex-disjoint copies of $F$ with index vector $\mathbf{v}$ for all $\bfv\in I$.
Note that $|V(\K^p)|=|V(\K^q)|= fC$.
When selecting any copy of $F$, we need to avoid at most $2f C$ vertices, which are incident to at most $2f C n^{f-1}\leq {\mu'} n^f/2$ copies of $F$. Therefore, there are at least $\mu' n^f/2$ choices for each copy of $F$ in 
$\K^p$ and $\K^q$ and in turn, at least $({\mu'} n^f/2)^{2C}$ choices for $\K^p$ and $\K^q$.

By \eqref{regroup11}, we have $\mathbf{i}_{\cP}(V(\K^q))+\bfu_1=\mathbf{i}_{\cP}(V(\K^p)) + \bfu_2$. 
This implies that we may write $V(\K^p)=\{y_1, \dots, y_{fC}\}$, $V(\K^q)= \{z_1, \dots, z_{fC}\}$ such that 
$y_1\in V_1$, $z_1\in V_2$, and for $i\ge 2$, $y_i$ and $z_i$ are from the same part of $\cP$ (and thus are $(F,\b, i_0)$-reachable to each other). 
We next select a reachable $(i_0 f-1)$-set $S_i$ for $y_i, z_i$ for $i\ge 2$ such that $S_2, \dots, S_{fC}$ are disjoint and also disjoint from $V(\K^p\cup \K^q)\cup\{x_1, x_2\}$. 
When selecting each $S_i$, we need to avoid at most constantly many vertices 
and thus there are at least $\frac{\b}2 n^{i_0 f-1}$ choices for each $S_i$. 
Finally, since $x_1$ and $y_1$ and respectively, $x_2$ and $z_1$ are $(F,\b, i_0)$-reachable, 
we can pick two disjoint $(i_0 f -1)$-sets $S_{1}, S_{0}$ such that $S_1$ is a reachable set for $x_1$ and $y_1$,  
$S_0$ is a reachable set for $x_2$ and $z_1$, and $S_1, S_0$ are disjoint from $V(\K^p\cup \K^q)\cup\{x_1, x_2\} \cup \bigcup_{i=2}^{fC}S_i$.
Again there are at least $\frac{\b}2 n^{i_0 f-1}$ choices for each of $S_{1}, S_{0}$. 
We claim that $A:=\bigcup_{i=0}^{fC} S_i \cup V(\K^p\cup \K^q)$ is a reachable set for $x_1$ and $x_2$. 
Indeed, $H[A\cup \{x_1\}]$ contains an $F$-factor because $H[S_i\cup\{z_i\}]$ for $i\ge 2$, $H[S_1\cup \{x_1\}]$, $H[S_{0}\cup \{z_1\}]$ and $\K^p$ all contain $F$-factors; on the other hand, $H[A\cup \{x_2\}]$ contains an $F$-factor because $H[S_i\cup\{y_i\}]$ for $i\ge 1$, $H[S_{0}\cup \{x_2\}]$ and $\K^q$ all contain $F$-factors.
Let $i_{1, 2} = i_0 f C +C+i_0$ and $\beta_{1, 2} = \left(\frac{\mu'}2\right)^{2C}\left(\frac{\b}2\right)^{fC+1}/(i_{1,2} f-1)!$. The procedure above provides at least
\[
\frac{\left(\frac{\mu'}2n^f\right)^{2C}\left( \frac{\b}2 n^{i_0k-1}\right)^{fC+1}}{(i_{1,2}f-1)!} = \beta_{1,2} n^{i_{1,2}f-1}
\]
reachable $(i_{1,2}f-1)$-sets for $x_1$ and $x_2$. 
\end{proof}

\subsection{Proof of Lemma \ref{lem.absorbing}}
The following simple fact will be used later for finding linear combinations of robust $K$-vectors.

\begin{fact}\label{fact}
Let $a,b,c\in \mathbb{Z}$. If $\gcd(a,b,c)=1$ and $\gcd(b-a,c-b)$ is odd, then
      $\gcd(a+b,a+c,b+c)=1$.
\end{fact}
\begin{proof}
Let $l=\gcd(a+b,a+c,b+c)$. Then $l\mid(b-a)$ and $l\mid(c-b)$ and consequently $l\mid \gcd(b-a,c-b)$.
Thus $l$ is odd. On the other hand, $l\mid2(a+b+c)$. Since $l$ is odd, it follows that $l\mid (a+b+c)$. Consequently, $l\mid a, l\mid b$ and $l\mid c$, which implies $l\mid\gcd(a,b,c)=1$, namely, $l=1$.
\end{proof}

\medskip

\begin{proof}[Proof of Lemma \ref{lem.absorbing}]
Fix $\delta \ge f(a,b,c)$ and $\r > 0$.
Let $\eta=\eta(\gamma)$ be the constant returned by Lemma \ref{erdos}.
In addition, assume that $\eta \le \min\{ 1/(2k), \gamma/4, \mu'_1/2 \}$, where $\mu'_1$ is the constant returned by Proposition \ref{supersaturation} with inputs $\mu=1/8$, $l_1=b$, and $l_2=c$.
Let $i_0 = 2^{\lfloor 1/(\delta+\r/2) \rfloor-1}$.
Let $\beta \ll\e \ll\r$ be the constants returned by Lemma \ref{lem:P}, and assume that $\e \le \eta^2/4$.
We pick $0<\mu \ll \e$ and let $\mu'$ the constant returned by Proposition \ref{supersaturation} with $\mu$, $l_1=a$, $l_2=b$, and $l_3=c$.
We apply Lemma \ref{lattticec} with $\beta$, $i_0$ and $\mu'$ and get $\beta'$ and $i_0'$.
Finally, we apply Lemma \ref{lemabsorbing} with $\beta'$, $\eta$ and $i_0'$, and get $\a>0$.

Let $n$ be sufficiently large and let $H$ be a $3$-graph on $n$ vertices such that  $\delta_1(H)\ge (\delta+\r) \binom n2$. It suffices to verify the assumptions ($\triangle$) and ($\diamondsuit$) in Lemma \ref{lemabsorbing} -- Lemma \ref{lemabsorbing} thus provides the desired vertex set $W$ (here $|W|\le \eta n\le \gamma n/4$).

If $\delta_1(H)\geq (6-4\sqrt{2}+\gamma)\binom n2$, then ($\diamondsuit$) holds by Lemma \ref{erdos}; 
otherwise by the definition of $f(a,b,c)$, we may assume that $a=1$ and $\delta_1(H)\geq (\frac14+\gamma)\binom n2$. By Proposition \ref{supersaturation}, there are at least $\mu'_1 (n-1)^{b+c} \ge \eta n^{k-1}$ copies of $K^{(2)}_{b, c}$ in the link graph\footnote{Given a $3$-graph $H$ and a vertex $v\in V(H)$, the link graph is defined as the graph with the vertex set $V(H)\setminus \{v\}$ and the edge set $\{S\setminus \{v\}: v\in S, S\in E(H)\}$.} of each vertex of $H$ (thus ($\diamondsuit$) holds).

In the rest of the proof we verify ($\triangle$) in cases depending on the type of $K=K_{a,b,c}$.
We first apply Lemma \ref{lem:P} to $H$ and obtain a partition $\cP=\{V_0, V_1, \dots, V_t\}$ of $V(H)$ such that $|V_0|\le 4\e n\le \eta^2 n$, $t\le \lfloor 1/(\delta+\r/2) \rfloor$, 
$|V_i|\ge \e^2 n$ and $V_i$ is $(K,\beta, i_0)$-closed in $H$ for all $i\in [t]$.
In particular, $t=1$ when $d=\gcd(b-a,c-b)$ is even (and $\delta \ge \frac12$); $t\le 2$ if $d\geq 3$ is odd (and $\delta \geq \frac49 $); $t\le 3$ if $d=1$ (and $\delta \geq \frac14$).

We are done if $t=1$. When $t\ge 2$, we consider $\mu$-robust edge vectors in $H$ with respect to the partition $\cP$.
By Lemma \ref{lattticec}, it suffices to verify the assumption in Lemma \ref{lattticec}, that is, $(1,-1)\in L_{\cP, K}^{\mu'}(H)$ when $t=2$ and respectively, $(1,-1,0)$, $(1,0,-1)$, $(0,1,-1)\in L_{\cP, K}^{\mu'}(H)$ when $t=3$.
For convenience, write
\[\t_1=(a,b+c), \t_2=(b,a+c), \t_3=(c,a+b), \t_4=(a+b+c,0)\]
and
\[\t_i'=(a+b+c,a+b+c)-\t_i \text{ for } 1\le  i\le 4.\]

\begin{claim}\label{clm:cross}
For any partition $\cP'=\{V_0, V', V''\}$ of $V(H)$ with $|V_0|\le 4\e n$ and $|V''|, |V'|\ge \e^2 n$, we have $(1,2)$ or $(2,1)\in I_{\cP'}^{3\mu}(H)$.
\end{claim}

\begin{proof}
Without loss of generality, assume that $|V'| \le n/2$. Fix $v\in V'$. We observe that $v$ is contained in at least $\e n^2$ crossing edges (those with index vector $(1,2)$ or $(2,1)$) -- otherwise $\delta_1(H)\leq \binom {n/2}2 + \e n^2 + |V_0|n < (\frac14 + \r)\binom n2$, contradicting our assumption on $\delta_1(H)$.
Hence $v$ is in at least $\e n^2/2$ edges with index vector $(1,2)$ or $\e n^2/2$ edges with index vector $(2,1)$.
Without loss of generality, assume that at least half of the vertices in $V'$ are in at least $\e n^2/2$ edges with index vector $(1,2)$. Thus the number of edges with index vector $(1,2)$ is at least 
$\frac12\e^2 n \cdot \e n^2/2 \ge 3\mu n^3$ as $\mu \ll \e$.
This means that $(1,2)\in I_{\cP'}^{3\mu}(H)$.
\end{proof}

\medskip
\noindent \textbf{Case 1: $K$ is of type $d\geq 3$ with $d$ odd.}

In this case, $\delta_1(H)\geq (\frac49+\r)\binom n2$. Thus $t = 2$ and $\cP=\{V_0, V_1, V_2\}$.
By Claim \ref{clm:cross}, without loss of generality, assume that $(1,2)\in I_{\cP}^{\mu}(H)$.
If $I_{\cP}^{\mu}(H)=\{(1,2)\}$, then  assume that $|V_2|=p n$ for some $0<p<1$.  The number of edges with index vector $(1,2)$ is at most
\[
|V_1| \binom{|V_2|}2< (1-p)p^2 n^3/2 \le \frac 49 \cdot \frac{n^3}{6},
\]
where the second inequality becomes an equality when $p=2/3$.
Thus, $e(H)\leq \frac49  \frac{n^3}{6} + 3\mu n^3 + |V_0| n^2 <\frac49\binom n3 + 5\e n^3 $ (where $3\mu n^3$ bounds the number of edges with other index vectors), contradicting our assumption on $\delta_1(H)$. Therefore, $|I_{\cP}^{\mu}(H)|\geq 2$ and there are $3$ possibilities: $I_{\cP}^{\mu}(H)\supseteq\{(1,2),(3,0)\}$, $I_{\cP}^{\mu}(H)\supseteq \{(1,2),(0,3)\}$ and $I_{\cP}^{\mu}(H)\supseteq \{(1,2),(2,1)\}$. By Proposition \ref{supersaturation},
\[
I_{\cP, K}^{\mu'}(H)\supseteq \{\t_1, \t_2, \t_3, \t_4 \} \text{  or  } I_{\cP, K}^{\mu'}(H)\supseteq \{\t_1,\t_2,\t_3,\t_4'\} \text{  or  } I_{\cP, K}^{\mu'}(H)\supseteq\{\t_1,\t_2,\t_3,\t_1',\t_2',\t_3'\},
\]
respectively. If $\{\t_1, \t_2, \t_3, \t_4 \}\subseteq I_{\cP, K}^{\mu'}(H)$,
\[
\t_4-\t_1=(b+c,-(b+c)), \ \t_4-\t_2=(a+c,-(a+c)), \, \t_4-\t_3=(a+b,-(a+b))\in L_{\cP, K}^{\mu'}(H).
\]
Since $K$ is of type $d\geq 3$ and $d$ is odd, Fact \ref{fact} implies that $\gcd(b+c,a+c,a+b)=1$ and hence  $(1,-1)=x(\t_4-\t_1)+y(\t_4-\t_2)+z(\t_4-\t_3)\in L_{\cP, K}^{\mu'}(H)$ for some integers $x, y, z$.
Otherwise $\{\t_1,\t_2,\t_3,\t_4'\}\subseteq I_{\cP,K}^{\mu'}(H)$ or $\{\t_1,\t_2,\t_3,\t_1',\t_2',\t_3'\}\subseteq I_{\cP,K}^{\mu'}(H)$, it is easy to see that in either case
\[
(a,-a), (b,-b), (c,-c) \in L_{\cP, K}^{\mu'}(H).
\]
Since $\gcd(a,b,c)=1$, we conclude that $(1,-1)\in L_{\cP, K}^{\mu'}(H)$.

\medskip
\noindent \textbf{Case 2: $K$ is of type $1$ and $t=2$.}

By Claim \ref{clm:cross}, without loss of generality, assume that $(1,2)\in I_{\cP}^{\mu}(H)$. By Proposition \ref{supersaturation}, we have
\[
\t_1, \t_2, \t_3\in I_{\cP, K}^{\mu'}(H),
\]
and thus
\[
\t_2-\t_1=(b-a,a-b), \t_3-\t_2=(c-b,b-c) \in L_{\cP, K}^{\mu'}(H).
\]
Since $K_{a,b,c}$ is of type $1$, namely, $\gcd(b-a,c-b)=1$, we conclude that $(1,-1)\in L_{\cP, K}^{\mu'}(H)$.

\medskip
\noindent \textbf{Case 3: $K$ is of type $1$ and $t=3$.}

If $(1,2,0)\in I_{\cP}^{\mu}(H)$, then the arguments in Case~2 show that $(1,-1, 0)\in L_{\cP, K}^{\mu'}(H)$. If we also have $(0, 1, 2)\in I_{\cP}^{\mu}(H)$, then $(0, 1, -1)\in L_{\cP, K}^{\mu'}(H)$. Consequently $(1, 0, -1)\in L_{\cP, K}^{\mu'}(H)$, and we are done. In general, let $T$ be the set of all vectors with three coordinates $0, 1, 2$ (in any order). If
\begin{equation}
\label{eq:T2}
\text{$I_{\cP}^{\mu}(H)$ contains two members of $T$ whose $0$'s are on different coordinates},
\end{equation}
then the arguments above show that $(1,-1,0), (0,1,-1), (1,0,-1)\in L_{\cP, K}^{\mu'}(H)$.

We claim that \eqref{eq:T2} holds if $(1,1,1)\notin I_{\cP}^{\mu}(H)$. In this case, we prove a stronger statement than \eqref{eq:T2}: \emph{for each $i\in [3]$, $I_{\cP}^\mu(H)$ contains a member of $T$ whose $i$th coordinate is positive}. Fix $i\in [3]$.
By applying Claim \ref{clm:cross} to $\cP'=\{V_0, V_i, V_{i+1}\cup V_{i+2}\}$ (the addition is modulo 3), we may assume that at least $3\mu n^3$ edges have index vector $(1, 2)$ with respect to $\cP'$. Since $(1,1,1)\notin I_{\cP}^{\mu}(H)$,  at most $\mu n^3$ of these edges have index vector $(1,1,1)$ with respect to $\cP$.  Thus, there exists $j\ne i$ such that at least $\mu n^3$ of these edges intersect $V_j$ with two vertices. This proves the desired statement.

What remains is the case when $(1,1,1)\in I_{\cP}^{\mu}(H)$. In this case, by Proposition \ref{supersaturation},
\[
(a,b,c),(b,a,c),(a,c,b),(b,c,a),(c,a,b),(c,b,a)\in I_{\cP, K}^{\mu'}(H).
\]
This implies $(y,-y,0), (0,y,-y), (y, 0, -y)\in L_{\cP, K}^{\mu'}(H)$ for all $y\in \{b-a,c-b\}$.
Since $\gcd(b-a,c-b)=1$, we derive that $(1,-1,0), (0,1,-1), (1,0,-1)\in L_{\cP, K}^{\mu'}(H)$.
\end{proof}

\section{Proof of the Almost Tiling Lemma}

\subsection{The weak regularity and cluster hypergraphs}

Let $H=(V,E)$ be a $3$-graph and let $V_1, V_2, V_3$ be mutually disjoint non-empty subsets of $V$. We denote the number of edges with one vertex in each $V_i$, $i\in [3]$ by $e(V_1,V_2,V_3)$, and the density of $H$ with respect to $(V_1, V_2, V_3)$ by
\[
d(V_1, V_2, V_3)=\frac{e(V_1,V_2,V_3)}{|V_1||V_2||V_3|}.
\]
The triple $(V_1,V_2,V_3)$ of mutually disjoint subsets $V_1, V_2, V_3\subseteq V$ is called \emph{$(\e,d)$-regular} for $\e>0$ and $d\geq 0$ if
\[
|d(A_1, A_2, A_3)-d|\leq \e
\]
for all triples of subsets $A_i\subseteq V_i$, $i\in [3]$, satisfying $|A_i|\geq \e |V_i|$. The triple $(V_1,V_2,V_3)$ is called \emph{$\e$-regular} if it is $(\e,d)$-regular for some $d\geq 0$. By definition,
if $A_i\subseteq V_i$, $i\in [3]$, has size $|A_i|\geq p|V_i|$ for some $p\ge \e$, then $(A_1, A_2, A_3)$ is $(\e/p,d)$-regular.


Let $H=(V,E)$ be an $n$-vertex $3$-graph, a partition of $V$ into $V_0, V_1, \dots, V_t$ is called an \emph{$(\e,t)$-regular partition} if 
\begin{itemize}
\item[(i)]$|V_1|=|V_2|=\cdots=|V_t|$ and $|V_0|\leq \e n$,
\item[(ii)] for all but at most $\e\binom t3$ sets $\{i,j,l\}\in \binom{[t]}3$, the triple $(V_{i},V_{j}, V_{l})$ is $\e$-regular.
\end{itemize}
We call $V_1,\dots, V_{t}$ \emph{clusters}. Given an $(\e,t)$-regular partition $\cP=\{V_0,V_1, V_2, \dots, V_t\}$ and $d>0$, the \emph{cluster hypergraph} $\R=\R(\e, d,\cP)$ is defined as the 3-graph whose vertices are $V_1,\dots, V_{t}$ and $\{V_i, V_j, V_l\}$ forms an edge of $\R$ if and only if $(V_i, V_j, V_l)$ is $\e$-regular and $d(V_i,V_j,V_l)\geq d$.

We need a simple corollary of the Weak Regularity Lemma, which is a straightforward extension of 
Szemer\'edi's regularity lemma for graphs \cite{Sze}.
The following proposition shows that the cluster hypergraph inherits the minimum degree of the original hypergraph. Since its proof is the same as that of \cite[Proposition 15]{BHS}, we omit the proof.

\begin{proposition}\cite{BHS}\label{lem:deg}
For $0<\e<d\ll\delta$ and $t_0\geq 0$ there exist $T$ and $n_2$ such that the following holds. Suppose $H$ is a $3$-graph on $n>n_2$ vertices with $\delta_1(H)\geq \delta\binom n2$. Then there exists an $(\e,t)$-regular partition $\cP$ with $t_0<t<T$ such that the cluster hypergraph $\R=\R(\e, d, \cP)$ satisfies $\delta_1(\R)\geq (\delta-\e-d)\binom t2$.
\end{proposition}

Next we show that every regular triple can be almost perfectly tiled by copies of $K_{a, b, c}$ provided the sizes of its three parts is somewhat balanced. 
\begin{proposition}\label{proalmost}
Let $a\le b\le c$ be integers, $0<2\e\leq d$, and $m$ be sufficiently large. Suppose $(V_1, V_2, V_3)$ is $(\e, d)$-regular, 
$|V_1|\leq|V_2|\leq |V_3|=m$, and
\begin{align}\label{triplecondition}
\frac{|V_1|}a\geq \frac{|V_2|}b\geq \frac{|V_3|}c.
\end{align}
Then there is a $K_{a,b,c}$-tiling on $V_1\cup V_2\cup V_3$ covering all but at most $\frac{c}{a}\e (|V_1|+|V_2|+|V_3|)$
vertices.
\end{proposition}

\begin{proof} We will greedily pick vertex disjoint $K_1,K_2,\dots$, $K_s$ until $|V_i \setminus \bigcup_{\ell=1}^s V(K_{\ell}) |<\e m$ for some $i\in [3]$, where each $K_{\ell}$ is a copy of $K_{a,b,c}$ or $K_{k,k,k}$ by the algorithm described below. 
This gives rise to a $K_{a,b,c}$-tiling because each copy of $K_{k,k,k}$ consists of three vertex disjoint copies of $K_{a,b,c}$. 
Our algorithm is as follows. 
For $i\in [3]$, let $U_i^0=V_i$. For $j\in [s]$, let 
\[
\left\{U_1^j, U_2^j, U_3^j \right\} = \left\{V_i\setminus \bigcup_{\ell=1}^jV(K_{\ell}): i\in [3] \right\}  \text{  such that  } |U_1^j|\le |U_2^j| \le |U_3^j|, 
\]
and $U^j = U_1^{j} \cup U_2^{j} \cup U_3^{j}$. 
In other words, $U_1^j, U_2^j, U_3^j$ are the subsets of $V_1, V_2, V_3$ obtained from removing the vertices of $K_1, \dots, K_j$ and arranged in the ascending order of size.
Suppose that we have already found $K_1,\dots, K_j$ and $|U_1^j|\ge \e m$.
We let $K_{j+1}$ be a copy of $K_{k,k,k}$ from $U^{j}$ if
\begin{align}\label{equaltriple}
|U_3^{j}|- |U_1^{j}| \le c - a;
\end{align}
otherwise we let $K_{j+1}$ be a copy of $K_{a,b,c}$ with $a$ vertices from $U_1^{j}$, $b$ vertices from $U_2^{j}$, and $c$ vertices from $U_3^{j}$. In either case this is possible because $|U_i^{j}|\geq \e m$ for $i\in [3]$; by the regularity, we have $d(U_1^{j}, U_2^{j}, U_3^{j})\geq d-\e\ge \e$ and
\[
e(U_1^{j},U_2^{j},U_3^{j})\geq \e |U_1^{j}||U_2^{j}||U_3^{j}|\geq \e^4 m^3.
\]
By Proposition \ref{supersaturation}, we can find a copy of $K_{k,k,k}$ or  $K_{a,b,c}$ from $U^{j}$. The algorithm terminates when $|U_1^s|< \e m$.  We need to show that  $|U^s|\le \frac ca \e (|V_1| + |V_2| + |V_3|)$. By \eqref{triplecondition}, $|V_1|+|V_2|+|V_3|\ge \frac kc m$ and thus $\frac ca \e (|V_1| + |V_2| + |V_3|)\ge \frac{k}{a}\e m$. So it suffices to show that $|U^s|\le \frac{k}{a}\e m$.

First, assume that \eqref{equaltriple} holds for some $0\le j<s$. In this case $K_{j+1}\cong K_{k,k,k}$ and
\[
|U_3^{j+1}| - |U_1^{j+1}| = |U_3^{j}|- |U_1^{j}| \leq c - a.
\]
Therefore $K_{\ell} \cong K_{k,k,k}$ for all $\ell > j$ and consequently  $|U_3^{s}| - |U_1^{s}|\le c -a$. Since $|U_1^s| < \e m$, it follows that $|U^s|< 3\e m +2(c-a)$. If $a=c$, then $|U^s| \le 3\e m=\frac ka \e m$ and we are done. Otherwise $\frac ka\ge 3+\frac 1a$. Since $m$ is large enough, it follows that $|U^s|< (3+\frac 1a)\e m\le \frac ka \e m$, as desired.

Second, assume that \eqref{equaltriple} fails for all $0\le j <s$. We claim that for all $0\le j\leq s$,
\begin{align}\label{regulartriple}
\frac{|U_1^j|}a \ge \frac{|U_2^j|}b \ge \frac{|U_3^j|}c.
\end{align}
This suffices because $|U_1^s|< \e m$ and \eqref{regulartriple} with $j=s$ together imply that $|U^s| \le (1+ \frac ba + \frac ca ) |U_1^s| < \frac ka \e m$.

Let us prove $\eqref{regulartriple}$ by induction. The $j=0$ case follows from \eqref{triplecondition} and the assumption $|V_1|\le |V_2|\le |V_3|$. Suppose that \eqref{regulartriple} holds for some $j\ge 0$.
By our algorithm, $K_{j+1}$ is a copy of $K_{a,b,c}$ with $a$ vertices from $U_1^{j}$, $b$ vertices from $U_2^{j}$, and $c$ vertices from $U_3^{j}$.  Let $\tilde{U}_i^j=U_i^{j}\setminus V(K_{j+1})$ for $i\in [3]$ and thus ${|\tilde{U}_1^{j}|}/a={|{U}_1^{j}|}/a-1$, ${|\tilde{U}_2^{j}|}/b= {|{U}_2^{j}|}/b-1$ and ${|\tilde{U}_3^{j}|}/c={|{U}_3^{j}|}/c-1$. 
By the inductive hypothesis,
\begin{equation}\label{tripleclaim1}
\frac{|\tilde{U}_1^{j}|}a \ge \frac{|\tilde{U}_2^{j}|}b \ge \frac{|\tilde{U}_3^{j}|}c.
\end{equation}

Since $|U_i^{j}| \ge \e m\ge b+c$ for all $i\in[3]$,
\[
b^2-a^2\leq (b-a)|U_1^{j}| \leq b|U_2^{j}|-a|U_1^{j}|
\]
and
\[
c^2-b^2\leq (c-b)|U_2^{j}| \leq c|U_3^{j}|-b|U_2^{j}|
\]
which implies that
\begin{align}\label{tripleclaim2}
\frac{|\tilde{U}_2^{j}|}a \ge \frac{|\tilde{U}_1^{j}|}b,\, \quad \text{ and } \quad \frac{|\tilde{U}_3^{j}|}b \ge \frac{|\tilde{U}_2^{j}|} c.
\end{align}

Now we separate cases according to the order of $|\tilde{U}_1^{j}|$, $|\tilde{U}_2^{j}|$ and $|\tilde{U}_3^{j}|$.  Since $|\tilde{U}_3^j| - |\tilde{U}_1^j| = |U_3^{j}| - |U_1^{j}| - (c-a) > 0$, we only have three cases.

\noindent\textbf{Case 1.} $|\tilde{U}_1^{j}| \le |\tilde{U}_2^{j}| \le |\tilde{U}_3^{j}|$. Then \eqref{regulartriple} for $j+1$ follows from \eqref{tripleclaim1} immediately.

\noindent\textbf{Case 2.} $|\tilde{U}_2^{j}| \le |\tilde{U}_1^{j}| \le |\tilde{U}_3^{j}|$. Together with \eqref{tripleclaim1} and \eqref{tripleclaim2}, we derive that
\[
\frac{|\tilde{U}_2^{j}|}a \ge \frac{|\tilde{U}_1^{j}|}b \ge \frac{|\tilde{U}_2^{j}|}b \ge \frac{|\tilde{U}_3^{j}|}c.
\]

\noindent\textbf{Case 3.} $|\tilde{U}_1^{j}| \le |\tilde{U}_3^{j}| \le |\tilde{U}_2^{j}|$. Together with \eqref{tripleclaim1} and \eqref{tripleclaim2}, we derive that
\[
\frac{|\tilde{U}_1^{j}|}a \ge \frac{|\tilde{U}_2^{j}|}b \ge \frac{|\tilde{U}_3^{j}|}b \ge \frac{|\tilde{U}
_2^{j}|}c.
\]
This implies that \eqref{regulartriple} holds for $j+1$ and we are done.
\end{proof}

When $a=b=c$, the proof of Lemma~\ref{lem:alm_til} is a standard application of the regularity method.
This was given implicitly in \cite{Khan1} and stated as \cite[Lemma 4.4]{LM1} without a proof. For completeness, we include the proof here.

\begin{proof}[Proof of Lemma~\ref{lem:alm_til} when $a=b=c$]
Let $0< 4\e = d \ll \min\{\r, \a\}$ and $t_0 = 1/\e$. 
Suppose $T$ and $n_2$ are the parameters returned by Proposition~\ref{lem:deg} with $\delta = 5/9 + \r$.
Let $H$ be a 3-graph on $n$ vertices with $\delta_1(H)\ge (\frac59 + \r) \binom{n}2$ for some sufficiently large $n\ge n_2$. 
We apply Proposition~\ref{lem:deg} and obtain an $(\e,t)$-regular partition $\cP$ with $t_0<t<T$ and a cluster hypergraph $\R=\R(\e, d, \cP)$ satisfying $\delta_1(\R)\ge (\frac59 + \r -\e-d)\binom t2$.  Note that each cluster is of size $n/t\ge n/T$. 
Suppose that $t\equiv r$ mod 3 for some $r \in \{0,1,2\}$. Let $\R'$ be the induced subgraph of $\R$ on clusters $V_{r+1}, \dots, V_t$.
Then $\delta_1(\R') \ge \delta_1(\R) - 2t \ge \left(\frac59 + \frac{\r}2 \right) \binom{t}2$.
We apply \cite[Theorem 6]{HPS}\footnote{We may alternatively use the exact result in \cite{Khan1, KOT}.}
to $\R'$ and get a perfect matching $M$. For each edge $e=\{V_{i}, V_{j}, V_{l}\}\in M$, Proposition~\ref{proalmost} provides a $K_{a, b, c}$-tiling that covers all but at most $\e(|V_i| + |V_j| + |V_l|)$ vertices of $V_i \cup V_j \cup V_l$. The union of these $K_{a, b, c}$-tilings covers all but at most
\[
|V_0| + \e(|V_1| + \cdots + |V_t|) + |V_1| + |V_2| \le 2\e n + 2n/t \le 4\e n\le  \alpha n
\]
vertices of $V(H)$, as desired.
\end{proof}

We assume that $a<c$ in the next two subsections.

\subsection{Fractional homomorphic tilings}

To obtain a large $K_{a,b,c}$-tiling in $H$ when $a<c$, we follow the idea of Bu\ss, H\`an and Schacht \cite{BHS} considering a fractional homomorphism from $K_{a,b,c}$ to the cluster hypergraph $\R$. Let us first define a fractional hom($K_{a,b,c}$)-tiling. As in previous sections, we simply write $K_{a, b, c}$ as $K$.

\begin{definition}\label{frac}
Given a 3-graph $H= (V, E)$, 
a function $h:V\times E\rightarrow [0,1]$ is called a fractional hom$(K)$-tiling of $H$ if
\begin{itemize}
 \item[(1)] $h(v,e)= 0$ if $v\not\in e$, 
 \item[(2)] $h(v)=\sum_{e\in E}h(v,e)\leq 1$,
 \item[(3)] for every $e\in E$ there exists a labeling $e=uvw$ such that $h(u,e)\le h(v,e)\le h(w,e)$ and
 \begin{align*}
 &\frac{h(u,e)}{a}\geq \frac{h(v,e)}b\geq \frac {h(w,e)}{c}.
 \end{align*}
 \end{itemize}
Given $e= uvw\in E$, we simply write $h(u,v,w)=(h(u,e), h(v,e),h(w,e))$. We denote by $h_{\min}$ the smallest non-zero value of $h(v,e)$
and by $w(h)$ the (total) weight of $h$:
\[
w(h)=\sum_{(v,e)\in V\times E}h(v,e).
\]
\end{definition}

For example, suppose that the vertex classes of $K$ are $X, Y , Z$ with $|X|=a$, $|Y|=b$ and $|Z|=c$. We obtain
a fractional hom$(K)$-tiling $h$ by letting $h(x,y,z)= \frac1{abc}(a, b, c)=(\frac1{bc},\frac1{ac},\frac1{ab})$ for every $xyz\in E(K)$ with $x\in X, y\in Y, z\in Z$.\footnote{In general, we write  $\lambda(x_1, x_2, x_3)= (\lambda x_1, \lambda x_2, \lambda x_3)$.}
 Then $w(h)=k$ (the largest possible) and $h_{\min}= \frac1{bc}$. We later refer to $(\frac1{bc},\frac1{ac},\frac1{ab})$ as the \emph{standard weight} of an edge of $K$ and refer to the function mentioned above as the \emph{standard weight function} on $K$. 

The following proposition shows that a fractional hom($K$)-tiling in the cluster hypergraph can be ``converted" to an integer $K$-tiling in the original hypergraph. 

\begin{proposition}\label{almost.frac}
Let $1\le a\le b\le c$ be integers.
Given $\e,\phi>0$, $d\ge 2\e/\phi$, and integer $T>0$, there exists $n_3\in \mathbb{Z}$ such that the following holds for all $0<t\le T$ and $n\ge n_3$.
Let $H$ be a $3$-graph on $n$ vertices with an $(\e, t)$-regular partition $\cP$ and a cluster hypergraph $\R=\R(\e, d,\cP)$. Suppose that there is a fractional hom$(K)$-tiling $h$ of $\R$ with $h_{\min}\geq \phi$. Then there exists a $K$-tiling of $H$ that covers at least $\left(1-2c\e/\phi \right) w(h) n/t $ vertices.
\end{proposition}

\begin{proof}
Let $\R'$ be the subhypergraph of $\R$ consisting of the hyperedges $e=uvw\in E(\R')$ with $h(u,e)$, $h(v,e)$, $h(w,e)\geq h_{\min}\ge \phi$. For each $u\in V(\R')$, let $V_u$ be the corresponding cluster of $H$. Since $\cP$ is an $(\e, t)$-regular partition, all the clusters have size $\ell$ for some $\ell \ge (1 - \e) n/t$.
In each $V_u$ we find disjoint subsets $V_u^e$ of size $h(u,e) \ell$ for all $e\in E(\R')$ with $u\in e$ -- this is possible because $\sum_{e\in E(\R')} h(u, e) \le 1$. 
Note that every edge $e=uvw\in E(\R')$ corresponds to an $(\e, d')$-regular triple $(V_u, V_v,V_w)$ for some $d'\ge d$. 
Hence for every $e=uvw\in E(\R')$, $(V_u^e,V_v^e,V_w^e)$ is $(\e/\phi, d')$-regular with at least $\phi \ell\ge (1 - \e)\phi n/t \ge (1 - \e)\phi n_3/T$ vertices in each part.
Because of Definition~\ref{frac} (3), the assumptions that $d\ge 2\e/\phi$ and $(1 - \e)\phi n_3/T$ is sufficiently large, we can apply Proposition \ref{proalmost} and obtain a $K$-tiling covering at least
\begin{align*}
\left(1-\frac{c}{a}\cdot \frac{\e}{\phi} \right) h(e) \ell \ge (1- c \e/\phi) h(e) (1 - \e) \frac{n}{t} \ge \left(1 - 2c\e/\phi \right)h(e) \frac{n}{t}
\end{align*}
vertices of $V_u \cup V_v \cup V_w$, where $h(e)=h(u,e)+h(v,e)+h(w,e)$.
Repeating this to all hyperedges of $\R'$, we obtain a $K$-tiling that covers at least
\[
\sum_{uvw\in E(\R')} \left(1- 2c\e/\phi \right) h(e) \frac{n}{t} = \left(1- 2c\e/\phi \right) w(h) \frac{n}{t}
\]
vertices of $H$.
\end{proof}

The key of the proof of Lemma~\ref{lem:alm_til} is that if a maximum $K$-tiling in $\R$ is not large enough, then we use the minimum degree condition to find a large fractional hom$(K)$-tiling of $\R$, which gives a large $K$-tiling in $H$ by Proposition~\ref{almost.frac}. The following two propositions show that we can find a fractional hom$(K)$-tiling $h$ on one or two copies of $K$ together with one or two vertices outside such that $w(h)$ is larger than the standard weight on these copies of $K$ alone.

Given a copy $K_1$ of $K$ and two vertices $u,u'\notin V(K_1)$, let $\L_1(K_1, u, u')$ denote the family of all $3$-graphs on $\{u, u'\}\cup V(K_1)$ whose edge set contains $E(K_1)$ and at least $a+1$ triples $u u' v$ with $v\in V(K_1)$.

\begin{proposition}\label{prop.frac.1copy}
Let $1\le a\le b\le c$ be integers with $a<c$ and $k=a+b+c$. 
Let $K_1$ be a copy of $K_{a, b, c}$ and let $u,u'\notin V(K_1)$ be two vertices.
For any $3$-graph $L\in \L_1(K_1, u, u')$, there is a fractional hom$(K)$-tiling $h$ of $L$ with $w(h)\geq k+\frac1{abc}$ and $h_{\min}\geq \frac1{bc^2}$.
\end{proposition}

\begin{proof}
Suppose the vertex classes of $K_1$ are $X, Y, Z$ with $|X|=a$, $|Y|=b$, and $|Z|=c$. Since $\deg(u u')\ge a+1=|X|+1$, we have $N(u u',Y\cup Z)\neq \emptyset$.

If there exists $z\in N(u u', Z)$, then we pick $x\in X$ and $y\in Y$ and assign weights $h(z, u, u')=(\frac1{bc},\frac1{ac},\frac1{ab})$,
\[
h(x, y, z)=\left(\frac1{bc}-\frac a{bc^2},\frac1{ac}-\frac1{c^2},\frac1{ab}-\frac1{bc} \right) = \frac{c-a}{abc^2} \left(a, b, c \right),
\]
and assign the standard weight to all other edges of $K_1$. Then $h$ is a fractional hom($K$)-tiling of $L$ with
$w(h)=k+\frac1{ab}+\frac1{ac}-\frac a{bc^2}-\frac1{c^2}\geq k+\frac1{abc}$
and $h_{\min} =\frac{c-a}{bc^2} \ge \frac1{bc^2}$.

Otherwise $N(u u', Z)=\emptyset$, then there exists $y\in N(u u',Y)$. First assume $a<b$. We assign $h(y, u, u')=(\frac1{bc},\frac1{ac},\frac1{ab})$,
\[
h(x, y, z)=\left(\frac1{bc}-\frac a{b^2c},\frac1{ac}-\frac1{bc},\frac1{ab}-\frac1{b^2} \right) = \frac{b-a}{ab^2c} \left(a, b, c \right)
\]
for some $x\in X$ and $z\in Z$, and the standard weight to all other edges. Then $h$ is a fractional hom($K$)-tiling with $w(h)= k+\frac1{ab}+\frac1{ac}-\frac a{b^2c}-\frac1{b^2}\geq k+\frac1{abc}$ and $h_{\min} =\frac{b-a}{b^2c} \ge \frac1{bc^2}$.
Second, we assume $a=b$. By the degree condition, we have $N(u u', X)\neq \emptyset$. Pick $x\in N(u u',X)$ and $z\in Z$.
By assigning $h(x, u, u')=h(y, u, u')=h(x, y, z)=(\frac1{2ac},\frac1{2ac},\frac1{2a^2})$  and the standard weight to all other edges, we get a fractional hom($K$)-tiling with $w(h)= k+\frac1{a^2}+\frac1{ac}-\frac1{2a^2}\geq k+\frac1{abc}$ and $h_{\min} = \frac1{2ac}\ge \frac{1}{bc^2}$ as $c\ge 2$.
\end{proof}

Given two vertex disjoint copies $K_1, K_2$ of $K$ and a vertex $u\not\in V(K_1)\cup V(K_2)$, let $\L_2(K_1, K_2, u)$ denote the family of all $3$-graphs on $\{u\} \cup V(K_1)\cup V(K_2)$ whose edge set contains $E(K_1)\cup E(K_2)$ and at least $\max\{a^2+2a(b+c), (a+b)^2\}+1$ triples $u v w$ with $v\in V(K_1)$ and $w\in V(K_2)$.

The following proposition shows that any 3-graph $L\in \L_2(K_1, K_2, u)$ has a hom$(K)$-tiling with weight greater than $2k$. In its proof we assign weights to an edge as follows. Suppose $0<\lambda\le 1$, then $(\frac ac \lambda, \frac bc \lambda, \lambda)$ satisfies (3) in Definition \ref{frac}. Furthermore, given $\mu_1, \mu_2 \ge 0$ such that $\frac ac \lambda +\mu_1\le \frac bc \lambda \le \lambda - \mu_2$, then $(\frac ac \lambda + \mu_1, \frac bc \lambda, \lambda-\mu_2)$ satisfies (3) in Definition \ref{frac} as well.

\begin{proposition}\label{propfractional}
Let $1\le a\le b\le c$ be integers with $a<c$ and $k=a+b+c$. 
Let $K_1, K_2$ be two vertex disjoint copies of $K_{a, b, c}$ and let $u\notin V(K_1)$ be a vertex.
For any $3$-graph $L\in \L_2(K_1,K_2,u)$, there exists a fractional hom$(K)$-tiling of $L$ with $w(h)\geq 2k+ \frac1{abc^2}$ and $h_{\min}\geq \frac1{bc^2}$.
\end{proposition}
\begin{proof}
For $i=1,2$, denote the vertex classes of $K_i$ by $X_i$, $Y_i$, $Z_i$ with $|X_i| =a$, $|Y_i|=b$, and $|Z_i|=c$.
Let $L_u$ be the bipartite graph on $V(K_1)\cup V(K_2)$ such that two vertices $v\in V(K_1)$ and $w\in V(K_2)$ are adjacent if and only if $uvw$ is an edge of $L$. 
Then $L_u$ satisfies the following properties.

\begin{enumerate}[(i)]
\item Since $\deg_{L}(u)\ge a^2+2a(b+c)+1$, $L_u$ must have an edge not incident to $X_1\cup X_2$.
\item Since $\deg_{L}(u)\ge (a+b)^2+1$, $L_u$ must have an edge incident to $Z_1\cup Z_2$.
\end{enumerate}

Let $\lambda=\frac1{abc}$.
Our proof is now divided into cases based on the values of $a, b$ and $c$.

\noindent \textbf{Case 1.} $b<c$.

First we assume that there is $z_1 z_2\in L_u$ for $z_1\in Z_1$ and $z_2\in Z_2$.
Let $x_i\in X_i$, $y_i\in Y_i$ for $i=1,2$.
In this case let $h(u, z_1, z_2)=(\lambda,\lambda,\lambda)$ and $h(x_1, y_1, z_1) = h(x_2, y_2, z_2) = (\frac1{bc}, \frac1{ac}, \frac1{ab} - \lambda)$.
Other edges of $K_1$ or $K_2$ receive the standard weight $(\frac1{bc},\frac1{ac},\frac1{ab})$. \emph{In the rest of the proof, any edge of $K_1$ or $K_2$ not specified receives the standard weight.}
Therefore we get a fractional hom$(K)$-tiling of $L$ with $w(h)= 2k + \lambda$ and $h_{\min}= \lambda$.
We thus assume $L_u[Z_1,Z_2]= \emptyset$ and proceed in two subcases.

\smallskip
\noindent \textbf{Case 1.1.} $a<b<c$.
We first assume that there exists $z_1 y_2\in L_u$ for $z_1\in Z_1$ and $y_2\in Y_2$.
Let $x_i\in X_i$ for $i=1,2$, $y_1\in Y_1$ and $z_2\in Z_2$.
In this case we let $h(y_2, z_1, u)=(\frac ac\lambda,\frac bc\lambda,\lambda)$, $h(x_1, y_1, z_1) = (\frac1{bc}, \frac1{ac}, \frac1{ab} - \frac bc \lambda)$, and $h(x_2, y_2, z_2) = (\frac1{bc}, \frac1{ac} - \frac ac \lambda, \frac1{ab} -\frac ab \lambda)$.
So we get a fractional hom$(K)$-tiling of $L$ with $w(h)= 2k+(1-\frac ab)\lambda\geq 2k+\frac1b\lambda$ and $h_{\min}= \frac ac\lambda$.

We thus assume that $L_u[Z_1, Y_2]= \emptyset$ and by symmetry, $L_u[Y_1, Z_2]= \emptyset$. By (i), it follows that $L_u[Y_1, Y_2]\neq \emptyset$. By (ii), without loss of generality, assume that $L_u[Z_1, X_2]\neq \emptyset$. Suppose $y_1y_2, z_1 x_2\in L_u$ with $y_1\in Y_1$, $y_2\in Y_2$, $z_1\in Z_1$ and $x_2\in X_2$.
Let $x_1\in x_1$ and $z_2\in Z_2$.
We let $h(y_1, y_2, u)=(\frac bc\lambda,\frac bc \lambda, \lambda)$, $h(x_2, u, z_1) = (\frac ac \lambda, \frac b{c}\lambda, \lambda)$,
$h(x_1, y_1, z_1) = \left(\frac1{bc}, \frac1{ac} - \frac bc \lambda, \frac1{ab} - \lambda \right)$, and $h(x_2, y_2, z_2) = \left(\frac1{bc} - \frac ac \lambda, \frac1{ac} - \frac bc \lambda, \frac1{ab} - \lambda \right)$.
So we get a fractional hom$(K)$-tiling of $L$ with $w(h)= 2k+\frac bc \lambda$ and $h_{\min}= \frac ac\lambda$.

\smallskip
\noindent \textbf{Case 1.2.} $a=b<c$.
We first assume that both $L_u[Z_1, X_2]\neq \emptyset$ and $L_u[Z_1, Y_2]\neq \emptyset$. Suppose $z_1x_2, z_1' y_2\in L_u$ with $z_1, z_1'\in Z_1$, $x_2\in X_2$ and $y_2\in Y_2$ (we may have $z_1=z_1'$).
We assign the weights $h(z_1, x_2, u) = h(z_1', y_2, u) = (\frac ac \lambda,\frac ac \lambda, \lambda)$.
If $a\ge 2$, then pick $x_1, x_1'\in X_1$, $y_1, y_1'\in Y_1$ and $z_2\in Z_2$, and assign $h(x_1, y_1, z_1)=h(x_1', y_1', z_1') = \left(\frac1{ac}, \frac1{ac}, \frac1{a^2} - \frac ac \lambda \right)$ and $h(x_2, y_2, z_2) = \left(\frac1{ac} - \frac ac \lambda, \frac1{ac} - \frac ac \lambda, \frac1{a^2} - \lambda \right)$.
If $a=1$, then pick $x_1\in X_1$, $y_1\in Y_1$ and $z_2\in Z_2$, and assign $h(x_2, y_2, z_2) = \left(\frac1{ac} - \frac{a}c \lambda, \frac1{ac} - \frac{a}c \lambda, \frac1{a^2} - \lambda \right)$ and 
\[
h(x_1, y_1, z_1)=h(x_1, y_1, z_1') = \begin{cases}
\left(\frac1{ac}, \frac1{ac}, \frac1{a^2} - \frac {2a}{c} \lambda \right)\text{ if }z_1=z_1', \\
\left(\frac1{ac}, \frac1{ac}, \frac1{a^2} - \frac {a}{c} \lambda \right) \text{ if }z_1\neq z_1'.
\end{cases}
\]
In all cases we get a fractional hom$(K)$-tiling of $L$ with $w(h)= 2k+ \lambda$ and $h_{\min}= \frac ac\lambda$.

We may thus assume that at least one of $L_u[Z_1, X_2]$ and $L_u[Z_1, Y_2]$ is empty, and by symmetry, at least one of $L_u[X_1, Z_2]$ and $L_u[Y_1, Z_2]$ is empty. Since $a=b$, $X_i$ and $Y_i$ ($i=1, 2$) play the same role.
Without loss of generality, assume that $L_u[Z_1, Y_2]=L_u[Y_1, Z_2]=\emptyset$.
Furthermore, we observe that $L_u[Z_1, X_2]\neq \emptyset$ and $L_u[X_1, Z_2]\neq \emptyset$ -- otherwise, as $L_u[Z_1, Z_2]= \emptyset$, it follows that $\deg_{L}(u)\le 4a^2 + ac < a^2 + 2a(b+c)$, a contradiction.

Suppose $z_1 x_2, x_1 z_2\in L_u$, where $z_1\in Z_1$, $x_2\in X_2$, $x_1\in X_1$, $z_2\in Z_2$.
By (i), there exists $y_1 y_2\in L_u$, where $y_1\in Y_1$, $y_2\in Y_2$ (see Figure 1).
We assign the weights $h(u, x_2, z_1) = h(u, x_1, z_2) = h(y_1, y_2, u) = (\frac ac \lambda,\frac ac \lambda, \lambda)$ and $h(x_1, y_1, z_1)=h(x_2, y_2, z_2)=(\frac1{ac} - \frac ac \lambda, \frac1{ac} - \frac ac \lambda, \frac1{a^2} - \lambda)$.
This gives a fractional hom$(K)$-tiling of $L$ with $w(h)= 2k+ \lambda+\frac {2a}c \lambda$ and $h_{\min}= \frac ac\lambda$.
Note that $h(u)=\frac {2a}c\lambda +\lambda=\frac{2a+c}{a^2 c^2}\le 1$ because $a\ge 1$ and $c\ge 2$.
Thus this weight assignment is possible.

\begin{figure}
\begin{center}
\begin{tikzpicture}
[inner sep=2pt,
   vertex/.style={circle, draw=blue!50, fill=blue!50},
   ]
\draw (0,2) ellipse (20pt and 10pt);
\draw (0,1) ellipse (20pt and 10pt);
\draw (0,0) ellipse (30pt and 15pt);
\draw (3,2) ellipse (20pt and 10pt);
\draw (3,1) ellipse (20pt and 10pt);
\draw (3,0) ellipse (30pt and 15pt);
\node at (-1,2) {$X_1$};
\node at (-1,1) {$Y_1$};
\node at (-1.3,0) {$Z_1$};
\node at (4,2) {$X_2$};
\node at (4,1) {$Y_2$};
\node at (4.3,0) {$Z_2$};
\node at (0,2) [vertex, label=left:$x_1$] {};
\node at (0,1) [vertex, label=left:$y_1$] {};
\node at (0,0) [vertex, label=left:$z_1$] {};
\node at (3,2) [vertex, label=right:$x_2$] {};
\node at (3,1) [vertex, label=right:$y_2$] {};
\node at (3,0) [vertex, label=right:$z_2$] {};
\draw (0,2) -- (3,0);
\draw (0,1) -- (3,1);
\draw (0,0) -- (3,2);
\draw[dashed] (0,0) -- (3,0);
\draw[dashed] (0,1) -- (3,0);
\draw[dashed] (0,0) -- (3,1);
\end{tikzpicture}

\caption{$L_u$ in the last subcase of Case 1.2.}
\end{center}
\end{figure}
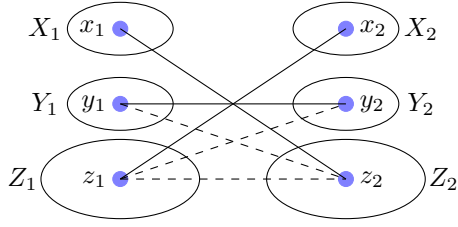

\medskip
\noindent \textbf{Case 2.} $a<b=c$.

Since $b=c$, $Y_i$ and $Z_i$ ($i=1, 2$) play the same role. Thus by (i), without loss of generality, assume that there exists $z_1z_2\in L_u$ for $z_1\in Z_1$ and $z_2\in Z_2$. Furthermore, generalizing (ii), we know that
there must be an edge incident to $Y_1\cup Y_2$ and without loss of generality, say that edge is incident to $Y_1$.
We now proceed with three cases.

\smallskip
\noindent \textbf{Case 2.1.} There exists $y_1 x_2\in L_u$ where $y_1\in Y_1$ and $x_2\in X_2$.
Pick $x_1\in X_1$ and $y_2\in Y_2$. 
We assign $h(u, z_1, z_2) = (\lambda, \lambda, \lambda)$, $h(x_2, y_1, u) = (\frac ac\lambda, \lambda, \lambda)$, $h(x_1, y_1, z_1) = \left(\frac1{c^2}, \frac1{ac} - \lambda, \frac1{ac} - \lambda \right)$, and $h(x_2, y_2, z_2) = \left(\frac1{c^2} - \frac ac \lambda, \frac1{ac} - \lambda, \frac1{ac} - \lambda \right)$.
Thus, we get a fractional hom$(K)$-tiling of $L$ with $w(h)= 2k+ \lambda$ and $h_{\min}= \frac ac\lambda$.

\smallskip
\noindent \textbf{Case 2.2.} There exists $y_1 y_2\in L_u$ where $y_1\in Y_1$ and $y_2\in Y_2$.
Pick $x_i\in X_i$ for $i=1,2$. 
We assign the weights $h(u, z_1, z_2) = h(u, y_1, y_2) = (\lambda, \lambda, \lambda)$ and
$h(x_1, y_1, z_1) = h(x_2, y_2, z_2) = \left(\frac1{c^2}, \frac1{ac} - \lambda, \frac1{ac} - \lambda \right)$
and get a fractional hom$(K)$-tiling of $L$ with $w(h)= 2k+ 2\lambda$ and $h_{\min}= \lambda$.

\smallskip
\noindent \textbf{Case 2.3.} There exists $y_1 z_2'\in L_u$ where $y_1\in Y_1$ and $z_2'\in Z_2$ (it is possible to have $z_2=z_2'$).
We assign the weights $h(z_2, u, z_1) = h(z_2', u, y_1) = (\frac ac\lambda, \lambda, \lambda)$.
Pick $x_1\in X_1$, $x_2\in X_2$ and distinct $y_2, y_2'\in Y_2$, which is possible as $b>a\ge 1$.
Let $h(x_1, y_1, z_1) = \left(\frac1{c^2}, \frac1{ac} - \lambda, \frac1{ac} - \lambda \right)$ and
$h(x_2, y_2, z_2) = h(x_2, y_2', z_2') = \left(\frac1{c^2}, \frac1{ac} - \frac ac \lambda, \frac1{ac} - \frac ac \lambda \right)$.
Thus, we get a fractional hom$(K)$-tiling of $L$ with $w(h)= 2k+ 2\lambda - \frac {2a}c \lambda \geq 2k+\frac 2c\lambda$ and $h_{\min}= \frac ac\lambda$.

In all cases we obtain a fractional hom$(K)$-tiling with $w(h)\ge 2k+\frac{\lambda}c = 2k+\frac{1}{abc^2}$ and $h_{\min}\ge \frac ac \lambda=\frac1{bc^2}$.
\end{proof}

\subsection{Proof of Lemma \ref{lem:alm_til} when $a<c$}
Let $H$ be a $3$-graph on $n$ vertices. Given $0\le \b \le1$, a $K$-tiling of $H$ is called \emph{$\b$-deficient} if it covers all but at most $\b n$ vertices of $V(H)$.

\begin{proposition}\label{deficient}
Given $0<d\le 3/5$ and $\b, \rho>0$, there exists an $n_0$ such that the following holds. If every $3$-graph $H$ on $n> n_0$ vertices with $\delta_1(H)\geq d\binom n2$ has a $\b$-deficient $K$-tiling, then every $3$-graph $H'$ on $n' >\max\{n_0, 5\}$ vertices with $\delta_1(H')\geq (d-\rho)\binom {n'}2$ has a $(\beta+2k\rho)$-deficient $K$-tiling.
\end{proposition}
\begin{proof}
Let $H'$ be a $3$-graph on $n'$ vertices with $\delta_1(H')\geq (d-\rho)\binom {n'}2$. By adding a set $A$ of $2\rho n'$ new vertices and all the triples of $V(H') \cup A$ that intersect $A$ as edges, we obtain a $3$-graph $H$ on $n=n'+ 2\rho n' $ vertices. Thus
\begin{align*}
\delta_1(H)
&= \delta_1(H') +2\rho n'(n'-1)+\binom{2\rho n'}2 \geq (d-\rho)\binom{n'}2+ 4\rho \binom{n'}{2} +\binom{2\rho n'}2.
\end{align*}
Note that $3\rho \binom{n'}{2} \ge 2d \rho n'^2$ because  $d\le 3/5$ and $n'\ge 5$. Thus, $\delta_1(H)\ge d\binom{n'}2+2d\rho {n'}^2+d\binom{2\rho n'}2=d \binom n2$. By assumption, $H$ has a $\b$-deficient $K$-tiling. After removing at most $2\rho n'$ copies of $K$ that intersect $A$, we obtain a $(\beta+2k\rho)$-deficient $K$-tiling of $H'$.
\end{proof}

\medskip
\begin{proof}[Proof of Lemma \ref{lem:alm_til} when $a<c$]
Since $a<c$, we have $k\ge 4$.
Let $\delta=\max\{1-(\frac {b+c}{k})^2,(\frac {a+b}{k})^2\}$.
Since $a\le b\le c$, it follows that $\delta \le \max\{ 5/9, 4/9\} = 5/9$.
Without loss of generality, assume that $0<\r\le \min\{3/5 - \delta, 2\delta, a/(3k)\}$.
Assume for a contradiction that there is an $\a$ such that for all $n_0$ there is some $3$-graph $H$ on $n>n_0$ vertices with $\delta_1(H)\geq (\delta+\gamma)\binom n2$ but which does not contain an $\a$-deficient $K$-tiling. Let $\a_0$ be the supremum of all such $\a$.

Let $\e\ll \r\a_0$. By the definition of $\a_0$, there is an integer $n_0$ such that
\begin{equation}
\label{eq:a0+e}
\text{all $3$-graphs $H$ on $n>n_0$ vertices with $\delta_1(H)\geq (\delta+\gamma)\binom n2$ have an $(\a_0+\e)$-deficient $K$-tiling.}
\end{equation}
We may also assume that $n_0$ is sufficiently large so that we can apply Proposition~\ref{supersaturation} with $r=3$, $m=1$, $l_1=a$, $l_2=b$, $l_3=c$ on $3$-graphs of order at least $\a_0 n_0/2$.
Our goal is to show that there exists an $n_1$ such that all $3$-graphs $H$ on $n>n_1$ vertices with $\delta_1(H)\geq (\delta+\gamma)\binom n2$ have an $(\a_0-\e)$-deficient $K$-tiling, thus contradicting the definition of $\a_0$.

Let $n_2$ and $T$ be the integers returned from Proposition~\ref{lem:deg} with inputs $\e$, $d=2b c^2 \e$, $t_0=\max\{n_0, k/\e\}$. 
Let $n_3$ be the integer returned from Proposition~\ref{almost.frac} with inputs $\e, \phi=1/(bc^2), d$ and $T$.
Let $n_1=\max\{n_0, n_2, n_3\}$ and let $H$ be a $3$-graph on $n>n_1$ vertices with $\delta_1(H)\geq (\delta+\gamma)\binom n2$. We assume that $H$ does \emph{not} contain an $(\a_0-\e)$-deficient $K$-tiling -- otherwise we are done.
After applying Proposition~\ref{lem:deg} to $H$ with the constants chosen above, we get
an $(\e,t)$-regular partition $\cP$ with $t_0<t<T$ and a cluster hypergraph $\R=\R(\e, d, \cP)$ on $t>t_0$ vertices with $\delta_1(\R)\geq (\delta+\gamma - (2bc^2+1)\e) \binom t2$. By \eqref{eq:a0+e} and assumption $\delta + \gamma \le 3/5$, we can apply Proposition \ref{deficient} and obtain an $(\alpha_0 + \e+2k(2bc^2+1)\e)$-deficient $K$-tiling of $\R$.
Let $\M=\{K_1,K_2,\dots,K_m\}$ be a largest $K$-tiling in $\R$
and let $U$ be the set of uncovered vertices.

\begin{claim}\label{gainweight}
Let $h$ be a fractional hom$(K)$-tiling of $\R$ with $h_{\min}\geq \frac1{bc^2}$. Then $w(h)<(1-\a_0+\sqrt\e/2)t\le  mk+\sqrt{\e} t$.
\end{claim}

\begin{proof}
We know that $|U|\le (\alpha_0 + \e+2k(2bc^2+1)\e)t \le (\alpha_0 + 5k b c^2\e)t$. As $\e \ll 1$,  it follows that
\[
mk+\sqrt\e t \ge (1 - \a_0 - 5kbc^2\e)t+\sqrt \e t\ge (1-\a_0+\sqrt\e/2)t.
\]
So it suffices to show that $w(h)<(1-\a_0+\sqrt\e/2)t$. Suppose this is not the case. By Proposition \ref{almost.frac}, there is a $K$-tiling of $H$ that covers at least
\begin{align*}
\left(1- 2bc^3 \e \right)w(h) \frac nt
\geq \left(1-2bc^3\e\right) (1 - \a_0 +\sqrt \e/2)t \frac nt \geq\left(1-\a_0+\e\right)n
\end{align*}
vertices (as $\e\ll 1$). Therefore it is an $(\a_0-\e)$-deficient $K$-tiling, contradicting our assumption on $H$.
\end{proof}

In the rest of the proof we will derive a contradiction to Claim \ref{gainweight}. Immediately Claim \ref{gainweight} implies that
\begin{equation}\label{leftover}
|U|\geq \frac {\alpha_0}2t
\end{equation}
otherwise $\mathcal M$ gives a fractional hom$(K)$-tiling $h$ with $w(h)= mk\ge (1- {\a_0}/2)t \ge (1-\a_0 +\sqrt\e/2)t$, as $\e \ll \a_0$.

Let $E_3=\{e\in E(\R): e\subseteq U\}$ and $E_2=\{e\in E(\R): |e\cap U|=2\}$.

\begin{claim}\label{edgeswithu}
 $|E_3|\leq \r \binom {|U|}3/2$ and $|E_2|\leq \delta\binom {|U|}2 mk$.
\end{claim}

\begin{proof}
By \eqref{leftover} and Proposition \ref{supersaturation}, we have $|E_3|\leq \r \binom {|U|}3/2$, as otherwise there exists a copy of $K$ in $U$, contradicting the maximality of $\mathcal M$.

Suppose, to the contrary, that $|E_2|> \delta\binom {|U|}2 mk$.
Let $\A$ be the set of all triples $i u u'$, $i\in [m]$, $u\ne u'\in U$ such that $u u'$ is adjacent to at least $a+1$ vertices in $K_i$. By the definition of $\L_1$, $i u u'\in \A$ if and only if $\R[V(K_i)\cup \{u, u'\}]\in \L_1(K_i, u, u')$. Let $\A_0$ be a largest matching in $\A$.
By the maximality of $\A_0$, for any $i\in [m] \setminus V(\A_0)$ and any $u\ne u' \in U\setminus V(\A_0)$,  at least $k-a$ vertices of $K_i$ are not adjacent to $u u'$.
Counting the number of non-edges $e\not\in E(\R)$ with $|e\cap U|=2$, we have
\[
(k-a)(m-|\A_0|)\binom{|U|-2|\A_0|}{2}\le \binom {|U|}2 mk - |E_2| < (1-\delta) \binom{|U|}2 m k.
\]
Since $(1-\delta) k \le (\frac{b+c}k)^2 k =\frac{(k-a)^2}k$, it follows that
\begin{equation}
\label{eq:A0}
(m-|\A_0|)\binom{|U|-2|\A_0|}{2}\le \frac{k-a}k m\binom {|U|}2.
\end{equation}
We claim that $|\A_0|\ge \r \a_0 m$. Indeed, \eqref{leftover} implies that $|U|\ge \a_0 t/2 \ge \a_0 mk/2\ge 2\a_0 m$ (as $k\ge 4$).
If $|\A_0|< \r \a_0 m$, then $m-|\A_0|\ge (1-\r \a_0)m$ and $|U|-2|\A_0|\ge |U| - 2\r \a_0 m\ge (1-\r)|U|$.
Thus \eqref{eq:A0} implies that
\[
\frac{k-a}k m\binom {|U|}2\ge (1-\r \a_0) m\binom{(1-\r)|U|}{2}\ge (1-\r \a_0)(1-2\r) m\binom{|U|}{2}> (1-3\r)m\binom{|U|}2
\]
contradicting $\r \le  \frac{a}{3k}$.
Now let $\A'\subseteq \A_0$ be of size $\r \a_0 m$.
By Proposition \ref{prop.frac.1copy}, for each member of $\A'$, there is a fractional hom($K$)-tiling $h'$ of $\R[V(K_i)\cup \{u, u'\}]$ with $w(h')\geq k+\frac1{abc}$ and $h'_{\min}\geq \frac1{bc^2}$.
This gives rise to a fractional hom($K$)-tiling $h$ of $\R$ with $h_{\min}\geq \frac1{bc^2}$ and $w(h)\geq mk+\gamma\alpha_0 m/(abc)$.

To complete the proof, we need a lower bound for $m$. Recall that $\delta_1(\R)\ge (1-\left(\frac{b+c}{k}\right)^2+\r - (2bc^2 +1)\e)\binom t2$. Thus if $|U|> \frac {b+c}k t$, then $\binom{|U|}2\ge (\frac{b+c}{k})^2 \binom{t}2 - t$ and
\[ 
\delta_1(\R[U])\ge \delta_1(\R) - \binom t2 + \binom{|U|}2 > (\r - (2bc^2 +1)\e )\binom t2 - t \ge  \frac{\r}2 \binom t2,
\]
where the last inequality holds because $t\ge t_0 \ge 1/\e$.
This implies that $|E_3|> \frac13|U| \r\binom t2/2>\r \binom{|U|}3/2$, contradicting the first part of Claim~\ref{edgeswithu}. Therefore $|U|\le \frac {b+c}k t$ and $|V(\M)|=mk\ge \frac{a}k t$, which gives $m\ge \frac a{k^2}t$. The fractional hom($K$)-tiling $h$ of $\R$ thus satisfies
\[
w(h)\geq mk+ \frac{\gamma\alpha_0 m}{abc} \ge mk + \frac{\r \a_0 t}{k^2 bc} > mk+\sqrt{\e}t,
\]
as $\e \ll \gamma\a_0$, contradicting Claim \ref{gainweight}.
\end{proof}

Let $\T$ be the set of all triples $uij$, $u\in U$, $i \ne j\in [m]$ such that there are at least $\delta k^2+1$ edges $u v w$ of $\R$ with $v\in V(K_i)$ and $w\in V(K_j)$.
Since $\delta k^2 +1=\max\{a^2+2a(b+c), (a+b)^2\}+1$, by the definition of $\L_2$, $uij \in \T$ if and only if 
 $\R[V(K_i) \cup V(K_j)\cup \{u\}] \in \mathcal{L}_2(K_i,K_j,u)$. Let $\T_0$ be a largest matching in $\T$.

\begin{claim}
$|\T_0|\geq \frac{\gamma\alpha_0}{6k}t$.
\end{claim}
\begin{proof}
We first derive a lower bound for $|\T|$ by considering $\sum_{u\in U}{\deg_\R(u)}$.
First partition the edges of $\R$ intersecting $U$ based on whether they contain one, two or three vertices of $U$.
Next we partition the edges $u x y$ of $\R$ with $u\in U$ and $x\in V(\K_i), y\in V(\K_j)$ (i.e., the edges of $\R$ with exactly one vertex in $U$) into three classes: (1) those with $i=j$, there are at most $\binom {k}2 m|U|$ such edges; (2) those with $i\neq j$ and $uij\notin \T$, there are at most $\delta k^2 |U|\binom m2$ such edges; (3) those with $i\neq j$ and $uij\in \T$, there are at most  $k^2|\T|$ such edges. Consequently,
\[
|U|\delta_1(\R) \leq \sum_{u\in U}{\deg_\R(u)} \leq 3|E_3|+2|E_2|+\binom {k}2 m|U|+\delta k^2 |U|\binom m2 + k^2|\T|.
\]
By Claim \ref{edgeswithu}, it follows that
 \begin{align*}
|U|\delta_1(\R) &\leq |U| \left(\frac{\gamma}2 \binom{|U|}2 +\delta |U|mk+\binom {k}2 m+\delta k^2\binom m2\right) +k^2 |\T |\\
& \le |U| \left( \delta \binom{|U|}2+\delta |U|mk +\delta k^2\binom m2 +\binom {k}2 m \right) +k^2 |\T | \quad \text{as } \r \le 2\delta \\
&\le |U| \left( \delta \binom{t}{2} + \frac{kt}{2} \right) +k^2 |\T|.
\end{align*}
On the other hand, $\delta_1(\R)\geq (\delta+\gamma - (2bc^2+1)\e)\binom t2$. Using $t\ge k/\e$ and $\e \ll \r$, we derive that
 $k^2 |\T| \geq |U| \cdot \frac{\gamma}2\binom t2$ or $|\T|\geq \frac{\gamma}{2 k^2} \binom {t}2 |U|$.

By the maximality of $\T_0$, all triples of $\T$ are covered by $V(\T_0)$. Since at most $\binom m2|\T_0|$ triples of $\T$ are covered by $V(\T_0)\cap U$ and at most $2|\T_0| (m-1) |U|$ triples of $\T$ are covered by $V(\T_0)\setminus U$, it follows that $|\T|\leq 2|\T_0| (m-1) |U| + \binom m2|\T_0|$. 
Since $mk-k+|U|=t-k$, we have $(mk - k)|U|\le (t-k)^2/4 \le \binom t2/2$. Consequently,
\[
|\T| \le \frac{|\T_0|}{k} \binom{t}{2} + \frac{ |\T_0|}{k^2}\binom {t}2.
\]
Together with
$|\T|\geq \frac{\gamma}{2 k^2} \binom {t}2 |U|$, we derive that $|\T_0|\geq \frac{\gamma|U|}{2k+2}\geq \frac{\r\a_0 t}{6k}$ using \eqref{leftover}.
\end{proof}

For every $u i j \in \T_0$, Proposition \ref{propfractional} provides a fractional hom($K$)-tiling $h$ of $\R[\{u\} \cup V(K_i) \cup V(K_j)]$ with $w(h)\geq 2k+ \frac1{abc^2}$ and $h_{\min}\geq \frac1{bc^2}$. Furthermore, for every $K_i\in \M$ with $i\not\in V(\T_0)$, we assign the standard weight on $K_i$. Hence, the union of all these fractional hom($K$)-tilings gives a fractional hom($K$)-tiling of $\R$ with $h_{\min}\geq \frac1{bc^2}$ and
\[
w(h)\geq \left(2k+ \frac1{abc^2}\right)|\T_0|+k(m-2|\T_0|)= mk+\frac1{abc^2}|\T_0|\ge mk+\sqrt{\e}t,
\]
as $\e \ll \gamma\a_0$, contradicting Claim \ref{gainweight}.
This completes the proof of Lemma \ref{lem:alm_til}.
\end{proof}

\section{Concluding Remarks}

In this paper, we investigate the minimum vertex degree conditions for tiling complete 3-partite 3-graphs $K$. Our result is best possible, up to the error term $\r n^2$. We remark that in some cases (\eg, $K= K_{1, 1, t}$ for $t\ge 2$) it seems possible to remove the error term and obtain exact results -- this was done for $K_{1, 1, 2}$ in \cite{Czy14, HZ3}. In general, in order to obtain an exact result, we need to have a stability version of the almost tiling lemma and a stability version of the absorbing lemma, together with an analysis of the 3-graphs that look like extremal examples. In many cases, when analyzing extremal examples, we need to know $\ex_1(n, K)$, the vertex-degree Tur\'an number for $K$, which is a challenging question in general. (The \emph{generalized Tur\'an number} $\ex_d(n, F)$ of an $r$-graph $F$ is the smallest integer $t$ such that every $r$-graph $H$ of order $n$ with $\delta_d(H) \ge t+1$ contains a copy of $F$.)

\smallskip

When proving the lower bound of Theorem~\ref{thm:main}, we introduced the covering barrier. In general, given an $r$-graph $F$, let $c_d(n, F)$ denote the minimum integer $c$ such that every $r$-graph $H$ of order $n$ with $\delta_d(H) \ge c$ has the property that every vertex of $H$ is covered by some copy of $F$. When $F$ is a graph, it is not hard to see that $c_1(n, F) = (1 - 1/(\chi(F) - 1) + o(1))n$: the lower bound follows from the $(\chi(F) -1)$-partite Tur\'an graph, and the upper bound can be derived after applying the Regularity Lemma to $V(H)\setminus \{v\}$ for an arbitrary vertex $v$ (see \cite{Zang_thesis} for details). Given an $r$-graph $F$, trivially
\begin{equation}
\label{eq:tau}
 \ex_d(n, F) < c_d(n, F) \le t_{d}(n, F).
\end{equation}
We know that $c_1(n, F) = \ex_1(n, F) + o(n)$ for all 2-graphs $F$. 
Construction~\ref{cons:t} and Lemma~\ref{erdos} together show that $c_1(n, K_{a, b, c}) = (6 - 4\sqrt{2} + o(1)) \binom n2$ if $2\le a\le b\le c$, while Theorem~\ref{thm:main} shows that $t_1(n, K_{a, b, c}) = (6 - 4\sqrt{2} + o(1)) \binom n2$ for certain $a, b, c$ (for example, $K_{2, 3, 6}$).  This shows that the upper bound for $c_d(n ,F)$ in \eqref{eq:tau} could be asymptotically tight.
For small 3-graphs $F$, determining $c_{2}(n, F)$ seems easier than determining $ \ex_2(n, F)$ or $t_{2}(n, F)$ (known as two difficult problems) -- see \cite{FaZh15} for recent progress. 

\smallskip

Let us give the following constructions of space barriers for complete $r$-partite $r$-graph tilings for arbitrary~$r$.
\begin{construction}
Fix positive integers $i < r$ and $a_1\le \cdots \le a_r$. Let $s= a_1+\cdots + a_r$ and $H_i$ be an $n$-vertex $r$-graph with $V(H_i)=A_i\cup B_i$ and $|A_i|=(a_1+\cdots + a_i) n/s-1$ such that $E(H_i)$ consists of all $r$-tuples containing at least $i$ vertices of $A_i$.
\end{construction}

To see why $H_i$ does not contain a $K_{a_1,\dots, a_r}$-factor, we observe that for each copy of $K_{a_1,\dots, a_r}$, at least $i$ color classes of it are subsets of $A_i$, and thus at least $a_1+\cdots + a_i$ vertices of it are in $A_i$. Since $|A_i| < (a_1+\cdots + a_i) n/s$, there is no $K_{a_1,\dots, a_r}$-factor of $H_i$. Thus the minimum $d$-degree threshold for a $K_{a_1,\dots, a_r}$-factor is greater than
$\max_{i\in [r-1]}\delta_{d}(H_i)$. Note that $\delta_{d}(H_{r-d+1})=0$ since any $d$-set in ${B_{r-d+1}}$ has degree zero. Thus, $\max_{i\in [r-1]}\delta_{d}(H_i)=\max_{i\in [r-d]}\delta_{d}(H_i)$. This means that there are $r-d$ space barriers, \eg,  there is only one construction for the $(r-1)$-degree case, and there are two constructions for the vertex degree threshold in $3$-graphs.

Since our main idea of proving Lemma \ref{lem:alm_til} (see also \cite{HZ1}) is to analyze the bipartite link graph of any uncovered vertex on two existing copies of $K$ in the partial tiling, new ideas are needed to attack the general vertex degree tiling problem. On the other hand, this also suggests that it seems possible to generalize Lemma \ref{lem:alm_til} to the one of tiling $r$-partite $r$-graphs under minimum $(r-2)$-degree.

\medskip
Another direction to extend the result of this paper is to study the minimum vertex degree conditions for non-complete 3-partite 3-graphs. 
Clearly if $F$ is a spanning subgraph of $K_{a,b,c}$ then $t_1(n, F)\le t_1(n, K_{a,b,c})$.
Note that there may be more than one choice of $K_{a,b,c}$ that contains $F$ as a spanning subgraph.
One of the referees pointed out the following example, which shows that 
\begin{equation}
\label{eq:last}
t_1(n, F)< \min t_1(n, K_{a,b,c})
\end{equation}
for some $F$, where the minimum is taken over all $K_{a,b,c}$ that contain $F$ as a spanning subgraph.
Indeed, take a copy of $K_{1,2,3}$ and denote $u$ as the vertex in the vertex class of size one.
Add three new vertices $x, y, z$ and new edges $u x y$ and $u x z$, and denote the resulting graph by $F$.
Then $K_{1,4,4}$ and $K_{1,3,5}$ are the only choices of $K_{a,b,c}$ that contain $F$ as a spanning subgraph.
By Theorem~\ref{thm:main}, $t_1(n, K_{1,4,4}) = (\frac49 + o(1))\binom{n}{2}$ and $t_1(n, K_{1,3,5}) = (\frac12 + o(1))\binom{n}{2}$.
On the other hand, $K_{2,7,9}$ has a perfect $F$-tiling. By Theorem~\ref{thm:main}, we have $t_1(n, F)\le t_1(n, K_{2,7,9}) = (6 - 4\sqrt{2} + o(1)) \binom n2$, which implies \eqref{eq:last}.

\section*{Acknowledgment}
We thank Richard Mycroft for valuable discussion.
We also thank two referees for their careful reading and detailed comments that improve the presentation of the paper. 
In particular, we are grateful to one referee for showing us the example that implies \eqref{eq:last}.

\bibliographystyle{plain}
\bibliography{Jan2014}

\end{document}